\documentclass[12pt]{article}

\usepackage[table]{xcolor}
\usepackage{amsmath, amsthm}
\usepackage{amsfonts,amssymb}
\usepackage{hyperref}
\usepackage{authblk}
\usepackage[all]{xy}
\usepackage{enumitem}
\usepackage{stmaryrd}
\usepackage[makeroom]{cancel}
\usepackage{bbm}
\usepackage{fancyvrb}
\usepackage[linesnumbered,lined,boxed]{algorithm2e}
\usepackage{xstring}
\usepackage{tikz}
\usetikzlibrary{intersections}

\renewcommand{\le}{\leqslant}
\renewcommand{\ge}{\geqslant}
\newcommand{\N}{\mathbb{N}}
\newcommand{\Z}{\mathbb{Z}}

\newcommand{\G}{\mathbb{G}}

\renewcommand{\P}{\mathcal{P}}

\renewcommand{\L}{\mathcal{L}}
\renewcommand{\O}{\mathcal{O}}

\newcommand{\Ld}{\L^\vee}

\newcommand{\Ker}{\operatorname{Ker}}

\newcommand{\rank}{\operatorname{rk}}

\newcommand{\ord}{\operatorname{ord}}

\newcommand{\Div}{\operatorname{Div}}
\newcommand{\Supp}{\operatorname{Supp}}
\newcommand{\Pic}{\operatorname{Pic}}
\newcommand{\Eff}{\operatorname{Eff}}

\renewcommand{\H}{\operatorname{H}}

\newcommand{\M}{\mathcal{M}}

\newcommand{\Set}[2]{\left\{ #1 \ \Big\vert \ #2 \right\}}

\newcommand{\torsor}{\times}
\newcommand{\Px}{\P^\torsor}
\newcommand{\Gm}{\G_m}
\newcommand{\Kx}{K^\torsor}
\newcommand{\dual}{\vee}
\newcommand{\Nm}{\mathcal{N}}
\newcommand{\defeq}{\overset{\text{def}}{=}}
\newcommand{\setmap}[4]{\begin{array}{rcl} #1 & \longrightarrow & #2 \\ #3 & \longmapsto & #4 \end{array}}
\newcommand{\EdC}{\Eff^\delta(C)}
\newcommand{\trans}{\top}
\renewcommand{\iff}{iff.\ }
\newcommand{\Lsymbol}{1}
\newcommand{\Rsymbol}{2}
\newcommand{\Ltimes}{\overset{\Lsymbol}{\otimes}}
\newcommand{\Rtimes}{\overset{\Rsymbol}{\otimes}}
\newcommand{\LTimes}{\bigotimes^\Lsymbol}
\newcommand{\RTimes}{\bigotimes^\Rsymbol}

\SetKw{True}{True}
\SetKw{False}{False}
\newcommand{\la}{\leftarrow}
\SetKw{FAIL}{FAIL}
\SetKw{Goto}{go to line}

\newcommand{\DivAdd}{\texttt{DivAdd}}
\newcommand{\DivSub}{\texttt{DivSub}}
\newcommand{\FnMul}{\texttt{SecMul}}
\newcommand{\AddFlip}{\texttt{AddFlip}}
\newcommand{\Negation}{\texttt{Negation}}
\newcommand{\MulFlip}{\texttt{MulFlip}}

\numberwithin{equation}{section}

\newtheorem{lem}[equation]{Lemma}

\theoremstyle{definition}
\newtheorem{de}[equation]{Definition}
\newtheorem{rk}[equation]{Remark}

\newtheorem{ex}[equation]{Example}

\usepackage{float}
\newfloat{Figure}{H}{Fig}[section]
\newfloat{Algorithm}{H}{Algo}[section]
\newfloat{Strategy}{H}{Strat}[section]
\newfloat{Table}{H}{Tab}[section]
\usepackage{fancybox,framed}

\makeatletter
\newcommand{\subjclass}[2][2020]{%
 \let\@oldtitle\@title%
 \gdef\@title{\@oldtitle\footnotetext{#1 \emph{Mathematics subject classification:} #2}}%
}
\newcommand{\keywords}[1]{%
 \let\@@oldtitle\@title%
 \gdef\@title{\@@oldtitle\footnotetext{\emph{Key words and phrases.} #1.}}%
}
\let\c@table\c@equation
\let\c@figure\c@equation
\makeatother
\makeatother

\title{Poincar\'e \`a la Makdisi}
\subjclass{
14H40, 
14Q05, 
14Q20. 
}
\author{Nicolas Mascot\thanks{\href{mailto:mascotn@tcd.ie}{mascotn@tcd.ie}}}
\affil{\scriptsize{Trinity College Dublin}}

\VerbatimFootnotes

\begin{document}

\maketitle

\begin{abstract}
We show how to extend Makdisi's algorithms to compute explicitly with the Poincar\'e torsor on the Jacobian of an algebraic curve.
\end{abstract}

\renewcommand{\abstractname}{Acknowledgements}
\begin{abstract}
The work presented here is motivated by the collaboration~\cite{QuadChab}. The author especially thanks Guido Lido for bringing the theory of the Poincar\'e torsor to the author's attention and suggesting to make it algorithmic, and Davide Lombardo for his insightful suggestions.
\end{abstract}

\textbf{Keywords:} Poincar\'e bundle, Jacobian, algorithm.


\section{Introduction}\label{sect:intro}

Let~$A$ be an Abelian variety over a field~$K$, and let~$A^\dual$ be its dual, so that~$A^\dual$ parametrises isomorphism classes of algebraically trivial line bundles over~$A$. The \emph{Poincar\'e bundle} on~$A$ is the unique line bundle~$\P$ on~$A \times A^\dual$ such that
\begin{itemize}
\item for all~$y = [\L] \in A^\dual$,~$\P_{\vert A \times \{ y \}} \simeq \L$,
\item and~$\P_{\vert 0 \times A^\dual} \simeq \O_{A^\dual}$ is trivial.
\end{itemize}

\begin{ex}\label{ex:poincare_ell}
If~$A=E$ is an elliptic curve, then~$E \simeq E^\dual$ via~$P \mapsto \O_E(P-O)$ where~$O \in E$ is the group identity, so we may identify~$E \times E^\dual$ with~$E \times E$. One then checks easily that~$\P = \O_{E \times E}(\Delta - E \times \{ O \} - \{ O \} \times E)$, where~$\Delta$ denotes the diagonal divisor on~$E \times E$.
\end{ex}

In this article, we restrict ourselves to the case where~$A=J$ is the Jacobian of a ``nice''\footnote{In this article, by ``nice'' we mean nonsingular, complete, and geometrically integral.} algebraic curve~$C$ defined over~$K$ such that~$C(K) \neq \emptyset$, so that we may identify~$J(K)$ with~$\Pic^0(C/K)$. Thanks to the canonical polarisation~$\pi : J^\dual \simeq J$, as in Example~\ref{ex:poincare_ell} we may identify~$\P$ with a line bundle on~$J \times J$, which is sometimes called the \emph{Mumford bundle}. It has the property that for all~$x \in J(K)$,~$\P_{\vert \{ x \} \times J} \simeq \P_{\vert J \times \{ x \} }$ represents the line bundle on~$J$ corresponding to~$x$ via~$\pi$, and it may be described more explicitly as follows. Let~$x=[D] \in J(K)$ be represented by a divisor~$D \in \Div^0(C)$, and similarly let~$y = [E] \in J(K)$. Then~\cite[Remark 6.3.12]{EdiLido} the stalk~$\P_{x,y}$ of~$\P$ at~$(x,y) \in J \times J$ is
\begin{equation} \P_{x,y} = \Nm_D\big(\O_C(E)\big). \label{eqn:descrP} \end{equation}
This description, along with the canonical isomorphisms
\[ \Nm_D\big(\O_C(E)\big) \otimes \Nm_{D'}\big(\O_C(E)\big) \simeq \Nm_{D+D'}\big(\O_C(E)\big) \]
and
\[ \Nm_D\big(\O_C(E)\big) \otimes \Nm_D\big(\O_C(E')\big) \simeq \Nm_D\big(\O_C(E+E')\big), \]
show the existence of commutative monoid laws
\begin{equation} \P_{x,y} \times \P_{x',y} \longrightarrow \P_{x+x',y} \label{eqn:Lmonoid} \end{equation}
and
\begin{equation} \P_{x,y} \times \P_{x,y'} \longrightarrow \P_{x,y+y'}. \label{eqn:Rmonoid} \end{equation}
These laws are called \emph{partial} laws because of the restriction on their domains, i.e. the condition that the~$y$- (respectively,~$x$-) coordinates of their arguments must agree.

In order to turn these monoid laws into group laws, we introduce the Poincar\'e \emph{torsor}~$\Px$, which is defined as~$\P$ with the~$0$ section removed. Thus, if we keep the same notations as in the previous paragraph and further assume that~$D$ and~$E$ are disjoint so that the rational section~$1$ of~$\O_C(E)$ is regular and non-vanishing around~$D$, then its norm~$\Nm_D\big( 1 \in \O_C(E) \big) \in \Px_{x,y}$ defines a point on the stalk of~$\Px$ at~$(x,y) \in J \times J$, and since the stalks of~$\Px$ are~$\Gm$-torsors, we have
\begin{equation}\Px_{x,y} = \Set{\lambda \cdot \Nm_D\big( 1 \in \O_C(E) \big)}{\lambda \in \Kx}. \label{eqn:Pxstalk} \end{equation}
Furthermore, the partial monoid laws~\eqref{eqn:Lmonoid} and~\eqref{eqn:Rmonoid} restrict to partial \emph{group} laws
\begin{equation} \Px_{x,y} \times \Px_{x',y} \longrightarrow \Px_{x+x',y} \label{eqn:Lgroup} \end{equation}
and
\begin{equation} \Px_{x,y} \times \Px_{x,y'} \longrightarrow \Px_{x,y+y'}, \label{eqn:Rgroup} \end{equation}
henceforth referred to respectively as the \emph{left} and \emph{right} partial group laws, and which make~$\Px$ the \emph{universal biextension} of~$J \times J$ by~$\Gm$, cf.~\cite[\textsection 2]{EdiLido}.

\paragraph*{Notation.} We will denote these partial group laws \emph{multiplicatively}, by
\[ \Ltimes : \Px_{x,y} \times \Px_{x',y} \longrightarrow \Px_{x+x',y} \]
and
\[ \Rtimes : \Px_{x,y} \times \Px_{x,y'} \longrightarrow \Px_{x,y+y'}, \]
the numbers~$1$ and~$2$ referring to the factor of~$J \times J$ on which the addition takes place in~$J$. Given a point~$\pi \in \Px_{x,y}$ and an integer~$n \in \Z$, we will write~$\pi^{\Ltimes n} \in \Px_{nx,y}$ (respectively,~$\pi^{\Rtimes n} \in \Px_{x,ny}$) for the exponentiation of~$\pi$ to the~$n$ according to the left (respectively, right) partial group law; in particular, we have the inverses~$\pi^{\Ltimes -1} \in \Px_{-x,y}$ and~$\pi^{\Rtimes -1} \in \Px_{x,-y}$, which are such that~$\pi \Ltimes \pi^{\Ltimes -1} \in \Px_{0,y}$ corresponds to~$1 \in \Kx$ under the canonical trivialisation~$\Px_{0,y} \simeq \O^\torsor_{J,y} \simeq \Kx$, and similarly for~$\pi \Rtimes \pi^{\Rtimes -1} \in \Px_{x,0} \simeq \O^\torsor_{J,x} \simeq \Kx$. In contrast, given~$\pi \in \Px_{x,y}$ and a scalar~$\lambda$ in~$\Kx$, we will write
\[ \lambda \cdot \pi \in \Px_{x,y} \]
for the~$\Gm$-torsor action; thus for example~$(\lambda \cdot \pi)^{\Ltimes n} = \lambda^n \cdot \pi^{\Ltimes n}$ and~$(\lambda \cdot \pi)^{\Rtimes n} = \lambda^n \cdot \pi^{\Rtimes n}$. We hope that these notations offer an acceptable compromise between compactness and the need to tell apart, for example, left exponentiation, right exponentiation, and the~$\Gm$-torsor action.

\begin{ex}\label{ex:Lgroup} Let~$D$,~$D'$, and~$E \in \Div^0(C)$ represent points~$x$,~$x'$, and~$y \in J(K)$ respectively and be such that~$D$ and~$D'$ are disjoint from~$E$. We may apply the left partial group law to~$\pi = \Nm_D\big( 1 \in \O_C(E) \big) \in \Px_{x,y}$ and~$\pi' = \Nm_{D'}\big( 1 \in \O_C(E) \big) \in \Px_{x',y}$ to obtain an element of~$\pi \Ltimes \pi' \in \Px_{x+x',y}$ which, in view of~\eqref{eqn:Pxstalk}, will be of the form~$\lambda \cdot \Nm_{D''}\big( 1 \in \O_C(E) \big)$ for a~$\lambda \in \Kx$ which will be uniquely determined as soon as we have chosen a divisor~$D'' \sim D+D'$ (and disjoint from~$E$) to represent~$x+x' \in J(K)$. This raises the question of how to calculate~$\lambda$ in terms of~$D$,~$D'$,~$E$, and~$D''$; the goal of this article is to address such questions algorithmically.
\end{ex}


More specifically, the purpose of this article is to explain how elements of (the stalks of)~$\Px$ can be represented efficiently on a computer, and to describe algorithms implementing the left and right partial group laws. Since the left (respectively, right) partial group law takes the stalks at~$x$ and~$x'$ to the stalk at~$x+x'$ (respectively,~$y$ and~$y'$ to~$y+y'$), these laws extend, in a sense, the group law on~$J$; therefore, making them algorithmic requires one to make the group law of~$J$ algorithmic in the first place (so that we may, for example, compute~$D''$ in terms of~$D$ and~$D'$ in Example~\ref{ex:Lgroup}). Makdisi introduced in~\cite{Mak1},~\cite{Mak2} very nice and efficient algorithms to compute the group law of~$J$; and in this article, we will show how to extend these algorithms to~$\Px$.

This article is organised as follows. We begin by establishing a comparison formula for points in the same fibre of~$\Px$ in Section~\ref{sect:compare}. After this, we recall in Section~\ref{sect:rep} how points of~$J$ are represented in Makdisi's algorithms, and explain how to extend this representation to encode points of~$\Px$. After this, we review Makdisi's algorithms for the group law of~$J$ in Section~\ref{sect:Laws}, and we show how to turn them into algorithms for the partial group laws of~$\Px$, including variants such as fast exponentiation and bilinear combinations.  Finally, we show in Section~\ref{sect:naive} how to convert into Makdisi format a point of~$\Px$ specified in terms of divisors expressed as sums of points of~$C$.

All the algorithms presented in this article have been implemented and thoroughly tested by the author as part of the collaboration~\cite{QuadChab}. Similarly to Makdisi's algorithms, they rely on linear algebra of size~$O(g)$, so that, in particular, their complexity is~$O(g^\omega)$ operations in~$K$, where~$\omega \le 3$ denotes the exponent complexity of linear algebra. They are based on formulas (such as~\eqref{eq:Poincare_compare_0} below) whose form is easy to guess, but whose precise derivation is often tedious and sometimes tricky; we will therefore show in detail how these formulas are obtained and used. 

\section{Comparing points on~$\Px_{x,y}$}\label{sect:compare}

Let us fix two points~$x,y \in J(K)$. In~\eqref{eqn:Pxstalk}, in order to understand the stalk~$\Px_{x,y}$ by relying on the~$\Gm$-torsor structure, we picked divisors~$D,E \in \Div^0(C)$ (with non-intersecting supports) representing respectively~$x$ and~$y$ in order to construct a ``base point''~$\Nm_D\big( 1 \in \O_C(E) \big)$, the notation~$1 \in \O_C(E)$ meaning~$1$ is considered as a section of~$\O_C(E)$. Obviously, this base point is not canonical, since the choices of~$D$ and~$E$ such that~$[D]=x$ and~$[E]=y$ are not. Our purpose of this short section is thus to elucidate the dependency of this base point on~$D$ and~$E$.

First of all, in order to declutter notation, whenever~$D$ and~$E$ are divisors of degree 0 on~$C$ with non-intersecting support, from now on we will denote the corresponding ``base point'' by
\begin{equation} (D,E) \defeq \Nm_D\big( 1 \in \O_C(E) \big). \label{eqn:(DE)} \end{equation}
In order not to have to repeat the condition on the disjointness of the supports of our divisors, we adopt the following ``golden rule'' for the rest of this article:
\begin{equation}
\boxed{\text{The~$D$'s are disjoint from the~$E$'s.}} \label{eqn:PGoldRule} 
\end{equation}
This means that any divisor called~$D,D',D_3$, etc. will be assumed to have its support disjoint from any divisor called~$E,E'',E_0$, etc.

Our task is thus to understand the connection between~$(D,E)$ and~$(D',E')$ when~$D' \sim D$ and~$E' \sim E$. More specifically, since~$(D',E')$ then represents a point in the same stalk~$\Px_{x,y}$ as~$(D,E)$, there exists a unique~$\lambda \in \Kx$ such that~$(D',E') = \lambda \cdot (D,E)$, and we want to make this~$\lambda$ explicit.

We start by looking at what happens when one of the divisors moves in its linear equivalence class while the other one is fixed.

\begin{lem}\label{lem:Pppal}
Let~$D,E \in \Div^0(C)$ be two divisors of degree~$0$, and~$f$ a rational function on~$C$. Furthermore, write~$s \in \O_C(E)$ to mean that a section~$s$ is understood as a meromorphic section of~$\O_C(E)$, and denote by~$(s)_E$ its divisor as a section of~$\O_C(E)$. Then under the identification~\eqref{eqn:descrP}, we have
\begin{enumerate}
\item~$\Nm_{D+(f)}(s \in \O_C(E)) = f((s)_E) \cdot \Nm_{D}(s \in \O_C(E)).$\label{lem:Pppal:D}
\item~$\Nm_D(s \in \O_C(E+(f))) = f(D) \cdot \Nm_D(s \in \O_C(E))$.\label{lem:Pppal:E}
\end{enumerate}
\end{lem}

\begin{proof}
This follows from the properties of the~$\Nm$ functor. More specifically, part~\ref{lem:Pppal:D} is~\cite[Lemma 6.4.8]{EdiLido}, and part~\ref{lem:Pppal:E} follows from the functoriality of~$\Nm_D$.
\end{proof}

We deduce the following fundamental formula:

\begin{lem}\label{lem:Poincare_compare_0}
Let~$D_1,D_2,E_1,E_2 \in \Div^0(C)$ be divisors of degree~$0$ satisfying~\eqref{eqn:PGoldRule}, so that we have the two points~$(D_1,E_1)$ and~$(D_2,E_2)$ of the Poincar\'e torsor. Suppose furthermore that~$D_1 \sim D_2$ and~$E_1 \sim E_2$, so that these two sections actually only differ by a scalar, and that we know rational functions~$d$ and~$e \in K(C)^\torsor$ such that~$D_2=D_1+(d)$ and~$E_2=E_1+(e)$. Then the unique~$\lambda \in K^\times$ such that~$(D_2,E_2) = \lambda \cdot(D_1,E_1)$ is
\begin{equation} \lambda = d(E_2) \cdot e(D_1) = d(E_1) \cdot e(D_2). \label{eq:Poincare_compare_0} \end{equation}
\end{lem}

\begin{proof}
Applying both statements of Lemma~\ref{lem:Pppal}, we compute that
 \begin{align*}
 (D_2,E_2) \defeq& \Nm_{D_2}(1 \in \O_C(E_2)) \\
 =& \Nm_{D_1+(d)}(1 \in \O_C(E_2)) \\
 =& d(E_2) \cdot \Nm_{D_1}(1 \in \O_C(E_2)) \\
 =& d(E_2) \cdot \Nm_{D_1}(1 \in \O_C(E_1+(e))) \\
 =& d(E_2) \cdot e(D_1) \cdot \Nm_{D_1}(1 \in \O_C(E_1)), \\
 \end{align*}
 so that~$\lambda = d(E_2) \cdot e(D_1)$.
 
 A similar calculation handling the~$E$'s before the~$D$'s shows that
 \begin{align*}
 (D_2,E_2) \defeq& \Nm_{D_2}(1 \in \O_C(E_2)) \\
 =& \Nm_{D_2}(1 \in \O_C(E_1+(e))) \\
 =& e(D_2) \cdot \Nm_{D_2}(1 \in \O_C(E_1)) \\
 =& e(D_2) \cdot \Nm_{D_1+(d)}(1 \in \O_C(E_1)) \\
 =& e(D_2) \cdot d(E_1) \cdot \Nm_{D_1}(1 \in \O_C(E_1)), \\
 \end{align*}
 so that~$\lambda = d(E_1) \cdot e(D_2)$ as well.
\end{proof}

\begin{rk}
 The fact that~$d(E_2) \cdot e(D_1)=d(E_1) \cdot e(D_2)$ is actually an immediate consequence of Weil reciprocity.
\end{rk}

\section{Representing points on~$J$ and on~$\Px_{x,y}$}\label{sect:rep}

\subsection{Representing points on~$J$: review of Makdisi's algorithms}\label{sect:MakFramework}

Before we describe algorithms for~$\Px$, let us review Makdisi's algorithms~\cite{Mak1,Mak2} for Jacobians.

\subsubsection{Makdisi's framework}

We begin by picking a line bundle~$\L / C$ whose degree
\[ \delta \defeq \deg \L \]
satisfies
\begin{equation}
\delta \ge 2g+1,
\label{eqn:bound_delta}
\end{equation}
where~$g$ is the genus of~$C$. Makdisi's algorithms rely on performing linear algebra on subspaces of the global section spaces
\[ V_n \defeq \H^0(C,\L^{\otimes n}) \]
of the powers of~$\L$ for~$n=1, \cdots, 5$. When~$0 \neq s \in V_n$ is a nonzero section of~$\L^{\otimes n}$, we will write~$(s)_n$ for its divisor, which is an effective divisor of degree~$n \delta$. Including the subscript~$n$ in this notation will help clarify the calculations that underlie our algorithms for~$\Px$, and also provides useful disambiguation when~$s$ could be understood as a element of~$V_n$ for several values of~$n$, e.g. if~$\L$ has been chosen to be a sheaf of rational functions on~$C$. When~$D \in \Eff(C)$ is an effective divisor, we will also write
\[ V_n(-D) \defeq \H^0\big(C,\L^{\otimes n}(-D)\big) \subseteq V_n. \]

Next, in order to represent subspaces of~$V_n$ on a computer, we also fix points
\[ Z_1, \cdots, Z_{n_Z} \in C(K) \]
on~$C$, along with a local trivialisation of~$\L$ at each of the~$Z_i$. Thanks to these trivialisations, it makes sense to talk about the value of a section~$s \in V_n$ at a point~$Z_i$; we will denote this value simply by~$s(Z_i)$, without explicitly referencing the local trivialisation at~$Z_i$ so as not to clutter our notations. We require that the number~$n_Z$ of points that we pick satisfy
\begin{equation}
n_Z > 5 \delta; \label{eqn:bound_nZ}
\end{equation}
this ensures that the linear maps
\begin{equation}
\setmap{V_n}{K^{n_Z}}{s}{c_s \defeq \big(s(Z_1),\cdots,s(Z_{n_Z}) \big)}
\label{eqn:embedVn}
\end{equation}
are embeddings, since we are limiting ourselves to~$n \le 5$.

In what follows, we will think of the spaces~$V_n$,~$n \le 5$ as implicitly embedded into~$K^{n_Z}$ via~\eqref{eqn:embedVn}; however, given a section~$s \in V_n$, we will write~$c_s$ for the column vector~$\big(s(Z_1),\cdots,s(Z_{n_Z})\big) \in K^{n_Z}$ corresponding to~$s$ when we want to insist on the distinction between the mathematical object~$s$ and the computer data~$c_s$ handled by our algorithms.

\begin{rk}\label{rk:EnoughPts}
If~$C(K)$ is too small for sufficiently many points~$Z_i$ to be found easily, we can simply replace~$K$ with a finite extension~$K'$ such that~$C(K')$ is large enough; in this case, we can still compute in~$J(K)$ by working with data over~$K'$ (which will presumably be slower than over~$K$). We mention that there are variants of Makdisi's algorithms which do not even require~$C(K)$ to be non-empty, cf.~\cite[Section 2]{Mak2}.
\end{rk}

\begin{rk}\label{rk:EnoughPtsDisjoint}
We also note that for our implementation of~$\Px$ to be efficient, we actually require~$K$ to be large enough that the divisor of two random elements of~$V_n$,~$n \le 5$, be disjoint with very high probability. In theory, this too could force us to enlarge~$K$ if~$K$ is a finite field. We refer to~\cite{Bruin} for a detailed analysis of this probability which we will not carry out here; we will instead content ourselves with noting that in practice, this condition seems to be satisfied as soon as~$C(K)$ is large enough in the sense of Remark~\ref{rk:EnoughPts}.
\end{rk}


\subsubsection{Representing divisors}\label{sect:Mak_RepDivs}

We represent an effective divisor~$D \in \Eff(C)$ by the subspace~$V_n(-D)$ of~$V_n$ for some~$n \le 5$ large enough in comparison with~$g$ and~$\deg D$ for Riemann-Roch to guarantee that this representation is faithful.
We then represent this subspace, or more generally any subspace~$W \subseteq V_n$, by the~$n_Z \times d$ matrix
\[ M_W \defeq \begin{pmatrix} w_1(Z_1) & \cdots & w_d(Z_1) \\ \vdots && \vdots \\ w_1(Z_{n_Z}) & \cdots & w_d(Z_{n_Z}) \end{pmatrix} \]
where~$d = \dim W$ and~$w_1,\cdots,w_d$ is any~$K$-basis of~$W$.

We will occasionally need to represent~$W$ alternatively by an~$(n_Z-d)\times n_Z$ matrix~$L_W$ whose rows are linear equations describing the image of~$W$ by~\eqref{eqn:embedVn} as a subspace of~$K^{n_Z}$; converting one representation of~$W$ into the other is straightforward.

\begin{rk}
We will see that Makdisi's algorithms ultimately rely on linear algebra of size~$O(n_Z) \times O(d_0)$ (cf.\ for example the matrices~$M_W$ above). By taking~$d_0$ and~$n_Z$ to be both~$O(g)$ (but still satisfying~\eqref{eqn:bound_delta} and~\eqref{eqn:bound_nZ}, of course), we see that their complexity is therefore~$O(g^\omega$), where~$\omega \le 3$ is an exponent for linear algebra complexity. The extensions of these algorithms to~$\Px$ that we will present in this article retain this property.
\end{rk}

\subsubsection{Representing points of~$J$}\label{sect:MakRepPt}

We can then represent any point~$x \in J(K)$ as follows. Riemann-Roch and~\eqref{eqn:bound_delta} ensure that~$x$ can be represented as~$x = [\Ld(D)]$ for some effective divisor~$D \in \EdC$ of degree~$\delta$, where~$\Ld$ denotes the dual of~$\L$; we faithfully encode this divisor~$D$ as the subspace~$V_2(-D)$ of~$V_2$, which we in turn represent by a matrix~$M_{V_2(-D)}$. In order to declutter notations, we write~$M_D \defeq M_{V_2(-D)}$.

We observe that this representation of~$x$ is doubly non-unique, as it relies on the choice of~$D$ in its linear equivalence class and then on the choice of a basis of~$V_2(-D)$ to form the matrix~$M_D$.

In particular, the identity~$0 \in J(K)$ may be represented by~$M_D$ for any~$D \in \EdC$ of the form~$D=(s)_1$ for some~$0 \neq s \in V_1$. Note that the corresponding subspace~$V_2(-D)$ is then simply~$V_2(-D) = s \cdot V_1$.

\subsubsection{Initialising Makdisi's algorithms}\label{sect:MakInit}

Once we have made a choice of line bundle~$\L$ and of points~$Z_1,\cdots,Z_{n_Z}$, it is useful to initialise data encoding~$J$ by precomputing representations~$M_{V_n}$ and~$L_{V_n}$ of the section spaces~$V_n$ for all~$1 \le n \le 5$, since these are required to execute Makdisi's algorithms. For this, it is actually enough to compute~$M_{V_1}$, since we can then obtain~$M_{V_n}$ for higher~$n$ by applying repeatedly the~$\DivAdd$ routine introduced below, and then easily deduce the~$L_{V_n}$ from the~$M_{V_n}$. We give two examples showing how~$M_{V_1}$ can be computed explicitly in different settings.

\begin{ex}\label{ex:OC(D)}
Suppose that we have picked~$\L = \O_C(D_0)$ for some effective divisor~$D_0 \in \EdC$ of degree~$\delta \ge 2g+1$, which is a reasonable choice when~$C$ is given to us by a plane equation for instance. It is then a good idea to pick the points~$Z_1,\cdots,Z_n$ away from the support of~$D_0$, so as not to require nontrivial trivialisations of~$\L$ at the~$Z_i$. Then~$M_{V_1}$ is obtained by computing a basis of the Riemann-Roch space of~$D$, and by evaluating it at the~$Z_i$.
\end{ex}

\begin{ex}\label{ex:Modular}
Suppose that~$C=X(\Gamma)$ is a modular curve attached to a congruence subgroup~$\Gamma$ of level~$N \in \N$. It is then natural to pick a line bundle whose sections are modular forms of level~$\Gamma$. More specifically, since we need~$\deg \L \ge 2g+1$, we can take~$\L$ whose sections are~$V_1 = M_2(\Gamma)$ whenever~$C$ has at least 3 cusps; if~$C$ has fewer cusps, we instead take~$\L$ so that~$V_1=M_k(\Gamma)$ for large enough weight~$k \in \N$.

We then pick non-cuspidal points~$Z_i \in C(K)$ represented by an elliptic curve given by a Weierstrass equation and equipped with a~$\Gamma$-level structure. The choice of a Weierstrass equation determines a normalisation of the differential on the elliptic curve and thus a local trivialisation of~$\L$ at the point, by viewing the sections of~$\L$ as Katz modular forms.

In order to compute~$M_{V_1}$, it remains to explicitly evaluate a basis of~$V_1$ at such points. Assuming that very minor hypotheses are satisfied, Makdisi shows in~\cite{ModuliFriendly} that for fixed vectors~$v,w \in (\Z/N\Z)^2$,
\[ \lambda_{v,w} : \big( E, \ \beta : (\Z/N\Z)^2 \overset{\sim}{\rightarrow} E[N] \big) \longmapsto \begin{array}{c} \text{Slope of the line} \\ \text{joining~$\beta(v)$ to~$\beta(w)$} \\ \text{on the Weierstrass model} \end{array} \]
is a modular form of weight~$1$ and level~$\Gamma(N)$, and that every modular form of level~$\Gamma$ and weight~$\ge 2$ can be expressed as an isobaric polynomial in the~$\lambda_{v,w}$ for various~$v,w \in (\Z/N\Z)^2$. This beautiful result thus makes it possible to compute~$M_{V_1}$ explicitly and thus to construct a model of~$J$ without relying on plane equations nor on~$q$-expansions; see~\cite{MakMod} for full details in the case~$\Gamma_1(N) \le \Gamma \le \Gamma_0(N)$.
\end{ex}

\subsubsection{The three basic routines}

The elegance of Makdisi's algorithms is that they manage to reduce the problem of computing in~$J(K)$ to linear algebra on matrices such as~$M_D$ encoding subspaces of the~$V_n$. The following three routines serve as basic building blocks in these algorithms.

\begin{itemize}
\item Observe that under~\eqref{eqn:embedVn}, multiplication of sections corresponds to coordinatewise multiplication in~$K^{n_Z}$. Therefore, given a column vector in~$c_s \in K^{n_Z}$ representing a section~$s \in V_{n}$ and a matrix~$M_{V_{n'}(-D)}$ encoding an effective divisor~$D$, if~$n+n' \le 5$, then by multiplying coordinatewise each column of~$M_{V_n'(-D)}$ by~$c_s$, we obtain a matrix~$M_{V_{n+n'}(-(s)_n-D)}$ encoding~$(s)_n+D$. We call this routine~$\FnMul$.
\item Whenever~$n_1, n_2 \in \N$ satisfy~$n_1+n_2 \le 5$ and~$D_1, D_2 \in \Eff(C)$ are such that
\begin{equation}
\begin{array}{c}
\text{Both } n_1 \delta - \deg D_1 = \deg \L^{\otimes n_1}(-D_1) \\
\text{and } n_2 \delta - \deg D_2 = \deg \L^{\otimes n_2}(-D_2) \\
\text{ are } \ge 2g+1,
\label{eqn:condition_DivAdd}
\end{array}
\end{equation}
the multiplication map
\[ V_{n_1}(-D_1) \otimes_K V_{n_2}(-D_2) \longrightarrow V_{n_1+n_2}(-D_1-D_2) \]
is surjective~\cite[Lemma 2.2]{Mak1}. As the target has known dimension by Riemann-Roch, it is thus easy to compute a representation~$M_{ V_{n_1+n_2}(-D_1-D_2)}$ of~$ V_{n_1+n_2}(-D_1-D_2)$ from~$M_{V_{n_1}(-D_1)}$ and~$M_{V_{n_2}(-D_2)}$ when~\eqref{eqn:condition_DivAdd} is satisfied; we call this routine
\[ \DivAdd : \left(M_{V_{n_1}(-D_1)}, M_{V_{n_2}(-D_2)} \right) \longmapsto M_{ V_{n_1+n_2}(-D_1-D_2)}, \]
since it corresponds to the addition of the divisors~$D_1$ and~$D_2$ under the encoding~$D \mapsto M_{V_n(-D)}$.

\item Similarly, we have a routine 
\[ \DivSub : \left( M_{V_{n+n'}(-D_1-D_2)}, M_{V_{n}(-D_1)} \right) \longmapsto M_{V_{n'}(-D_2)} \]
effecting subtraction of divisors, where as usual~$n+n' \le 5$,~$D_1, D_2 \in \Eff(C)$, and under the extra condition that~$n\delta-\deg D_1 \ge 2g$ so that~$V_n(-D_1)$ is basepoint-free in the sense that
\begin{equation}
\inf_{\substack{s \in V_n(-D_1) \\ s \neq 0}} (s)_n = D_1. \label{eqn:bpfree}
\end{equation}
In order to explain how it works, embed everything into~$K^{n_Z}$ by~\eqref{eqn:embedVn}, and let~$(t_i)_i$ be the~$K$-basis of~$V_n(-D_1)$ represented by the columns of~$M_{V_{n}(-D_1)}$; then by~\eqref{eqn:bpfree},
\begin{align*}
V_{n'}(-D_2) &= \Set{s \in V_{n'} }{ s \cdot V_n(-D_1) \subseteq V_{n+n'}(-D_1-D_2)} \\
&= \Set{ c_s \in K^{n_Z}}{s \in V_{n'} \text{ and } \forall i, \ c_s \odot c_{t_i} \subseteq V_{n+n'}(-D_1-D_2)}
\end{align*}
where~$\odot$ denotes coordinatewise multiplication in~$K^{n_Z}$. Therefore~$V_{n'}(-D_2)$ can be recovered as the solution space of the linear system obtained by stacking the linear equations contained in~$L_{V_{n'}}$ with those obtained, for each~$i$, by multiplying the~$j$-th column of~$L_{V_{n+n'}(-D_1-D_2)}$ by~$t_i(Z_j)$ for all~$j$, where the~$L$ matrices are as in Section~\ref{sect:Mak_RepDivs}. Indeed, for each~$i$, doing so turns the linear system~$L_{V_{n+n'}(-D_1-D_2)}$ whose solution space is~$V_{n+n'}(-D_1-D_2)$ into a linear system whose solution space is~$\Set{ c_v \in K^{n_Z}}{v \cdot t_i \in V_{n+n'}(-D_1-D_2)}$.
\end{itemize}

\subsubsection{First application: equality test and zero test}


We defer the presentation of the algorithm for the group law of~$J(K)$ to section~\ref{sect:Laws} below, as we will show there how to extend it to compute the partial group laws of~$\Px$; but in order to whet the reader's appetite, we now present an algorithm that circumvents the disadvantage of the representation of points of~$J$ by matrices~$M_D$ not being unique, by testing whether two such representations encode the same point of~$J$.

\begin{algorithm}[H]
\KwIn{Matrices~$M_{D}$ and~$M_{D'}$ representing the points~$x = [\Ld(D)]$ and~$x = [\Ld(D')] \in J(K)$, where~$D, D' \in \Eff^{\delta}(C)$.}
\KwOut{1 if~$x=x'$, 0 else.}
$0 \neq c_u \la$ a random vector in the column span of~$M_{D}$\; \label{alg:MakEq:u}
\tcp{$u \in V_2(-D) \leadsto (u)_2 = D+E$ for some~$E \in \EdC$}
$M_{u \cdot V_2(-D')} \la \FnMul(c_u,M_{D'})$\tcp*{Represents~$V_4(-D-D'-E)$}
$M_W \la \DivSub(M_{u \cdot V_2(-D')}, M_D)$\tcp*{Represents~$V_2(-D'-E)$} \label{alg:MakEq:W}
\Return~$\#$ columns of~$M_W$\;
\caption{Makdisi's equality test in~$J$}
\label{alg:MakEq}
\end{algorithm}

\begin{proof}
Indeed,~$\deg \L^{\otimes 2}(-D'-E) = 2\delta-\delta-\delta=0$, so~$V_2(-D'-E) \defeq \H^0\big(C,\L^{\otimes 2}(-D'-E)\big)$ has dimension~$\le 1$, with equality \iff~$\L^{\otimes 2}(-D'-E)$ is trivial in~$\Pic^0(C)$. But \[ \L^{\otimes 2}(-D'-E) \simeq \frac1u \L^{\otimes 2}(-D'-E) = \O_C(D-D'). \qedhere \]
\end{proof}

In particular, we could test whether~$M_D$ represents~$0 \in J(K)$ by applying this algorithm to~$M_D$ and a representation~$M_{D'}$ of~$0 \in J(K)$. However, it is more efficient to proceed as follows.

\begin{algorithm}[H]
\KwIn{A matrix~$M_{D}$ representing the point~$x = [\Ld(D)] \in J(K)$, where~$D \in \Eff^{\delta}(C)$.}
\KwOut{1 if~$x=0$, 0 else.}
$M_W \la \DivSub(M_{D}, M_{V_1})$\tcp*{Represents~$V_1(-D)$} \label{alg:MakZero:W}
\Return~$\#$columns of~$M_W$\;
\caption{Makdisi's zero test in~$J$}
\label{alg:MakZero}
\end{algorithm}

\begin{proof}
Indeed,~$\deg \L(-D) = \delta-\delta=0$, so~$V_1(-D) \defeq \H^0\big(C,\L(-D)\big)$ has dimension~$\le 1$, with equality \iff~$\L(-D)$ is trivial in~$\Pic^0(C)$.
\end{proof}

\subsection{Representing points on~$\Px_{x,y}$}\label{sect:repPx}

In the previous section, we used the language of line bundles to explain how we encode a point~$\Ld(D) \in J(K)$; but in~\eqref{eqn:(DE)}, our description of the stalks of~$\Px_{x,y}$ used exclusively the language of divisors. In order to translate between these two languages, we can fix a section~$0 \neq d_0 \in V_1 = \H^0(C,\L)$, so that division by~$d_0$ yields an isomorphism~$\L \simeq \O_C(D_0)$, where
\[ D_0 \defeq (d_0)_1 \in \EdC. \]
Thus, given two points~$x,y \in J(K) \times J(K)$ represented in Makdisi format respectively as~$M_D$ and~$M_E$ for some~$D,E \in \EdC$, we have~$x=[\Ld(D)]=[D-D_0]$ and~$y=\Ld(E) = [E-D_0]$, and it is therefore tempting to consider the base point~$(D-D_0,E-D_0)$ of~$\Px_{x,y}$.

However, this would be wrong, as the symbol~$(D-D_0,E-D_0)$ only makes sense when~$D-D_0$ is disjoint from~$E-D_0$, which is plainly not the case. In order to circumvent this issue, we fix \emph{another} section~$0 \neq e_0 \in V_1$, chosen such that its divisor
\[ E_0 \defeq (e_0)_1 \in \EdC \]
is disjoint from~$D_0$ in the spirit of our golden rule~\eqref{eqn:PGoldRule}. We thus also have~$y=[E-D_0]=[E-E_0]$, whence the point~$(D-D_0,E-E_0) \in \Px_{x,y}$, which is now well defined as long as our divisors respect the disjointness conditions~\eqref{eqn:PGoldRule}.
In order to declutter notations, we will henceforth write
\begin{equation} [D,E] \defeq (D-D_0,E-E_0) \defeq \Nm_{D-D_0}\big( 1 \in \O_C(E-E_0) \big). \label{eqn:[DE]} \end{equation}
\begin{rk}
The reader should take care not to confuse this new square-bracket notation~$[D,E]$, which involves effective divisors~$D$ and~$E$ of degree~$\delta$ (and implicitly the divisors~$D_0$ and~$E_0$), with the old round-bracket notation~$(D,E)$, which involves divisors~$D$ and~$E$ of degree 0.
\end{rk}

Therefore, in order to initialise data encoding~$\Px$, we must first precompute data encoding~$J$ as explained in Section~\ref{sect:MakInit}, and then fix two random sections~$d_0$ and~$e_0$ of~$\L$ whose respective divisors~$D_0$ and~$E_0$ are disjoint. 

More specifically, recall from Section~\ref{sect:MakInit} that we have precomputed a matrix~$M_{V_1}$ representing the global section space~$V_1$ of~$\L$. We can thus pick the sections~$d_0, e_0 \in V_1$ simply by choosing two random vectors~$c_{d_0}, c_{e_0}$ in the column span of~$M_{V_1}$. Since~$V_2(-D_0) = d_0 \cdot V_1$, we can then use the~$\FnMul$ routine to easily compute and store for later use a representation~$M_{D_0} = M_{V_2(-D_0)}$ of~$D_0 \defeq (d_0)_1$, and similarly for~$E_0 \defeq (e_0)_1$.
Indeed, we will see that these matrices~$M_{D_0}$ and~$M_{E_0}$ are required to compare points on the same fibre in~$\Px$ (cf. Remark~\ref{rk:compare_need_M0N0} below); for instance,~$M_{D_0}$ is what we need to apply Algorithm~\ref{alg:norm} below at~$D=D_0$, and similarly for~$M_{E_0}$. Furthermore, the matrices~$M_{D_0}$ and~$M_{E_0}$ are also what we need in order to check that~$D_0$ and~$E_0$ are disjoint with the following algorithm. Note that~$D_0$ and~$E_0$ are indeed likely disjoint since we assume that we are in the favourable situation described in Remark~\ref{rk:EnoughPtsDisjoint}; so if~$D_0$ and~$E_0$ are actually not disjoint, we simply start over with new random choices of~$d_0$ and~$e_0$.

\begin{algorithm}[H]
\KwIn{Two matrices~$M_{D_1}$,~$M_{D_2}$ representing divisors~$D_1, D_2 \in \EdC$.}
\KwOut{\True if~$D_1$ and~$D_2$ are disjoint; else \False.}
\For{$i=1$ \KwTo~$2$}{
$M_{V_3(-D_i)} \la \DivAdd(M_{D_i},M_{V_1})$\tcp*{Raise from~$V_2$ to~$V_3$}\label{alg:disjoint:raisen}
}
$M \la$ horizontal concatenation of~$M_{V_3(-D_1)}$ and~$M_{V_3(-D_2)}$\;
$r \la \rank M$\tcp*{$\dim V_3(-D_1) + V_3(-D_2)$}\label{alg:disjoint:rank}
\eIf{$r=3\delta+1-g$}{\Return \True\;}{\Return \False\;}
\caption{Divisor disjointness test (cf.~\cite[3.4]{Mak1})}
\label{alg:disjoint}
\end{algorithm}

\begin{proof}
Fix an integer~$n \le 5$, and consider the addition map
\[ A_n : V_n(-D_1) \times V_n(-D_2) \longrightarrow V_n. \]
By rank-nullity,~$\rank A_n$ is directly connected to~$\dim V_n(-D_1) \cap V_n(-D_2)$. But clearly~$V_n(-D_1) \cap V_n(-D_2) = V_n(-D)$, where~$D \defeq \sup(D_1,D_2) \in \Eff(C)$. We have~$\deg D \le 2 \delta$, with equality \iff~$D_1$ and~$D_2$ are disjoint, so we want to put our hands on~$\deg D$.

If~$n \ge 3$, then Riemann-Roch ensures that~$\dim V_n(-D) = n\delta - \deg D + 1 -g$, so that~$\deg D$ can be read off from~$\rank A_n$; in contrast, for~$n \le 2$,~$\dim V_n(-D)$ need not reflect~$\deg D$ accurately.

Since we are provided with the matrices~$M_{D_i} \defeq M_{V_2(-D_i)}$ representing the~$V_n(-D_i)$ for~$n=2$, we thus first use~$\DivAdd$ at line~\ref{alg:disjoint:raisen} to compute the~$V_n(-D_i)$ for~$n=3$, after which we compute~$r = \rank A_3$ at line~\ref{alg:disjoint:rank} and conclude.
\end{proof}

A point on~$\pi \in \Px_{x,y}$ will then be encoded by a triple
\begin{equation}
\pi = (M_D,M_E,\lambda), \label{eqn:codepi}
\end{equation}
where~$M_D$ is a matrix encoding~$x=[\Ld(D)]$,~$M_E$ is a matrix encoding~$y=\Ld(E)$, and~$\lambda \in \Kx$ is the unique scalar such that~$\pi = \lambda \cdot [D,E]$, where the divisors~$D$,~$E$,~$D_0$, and~$E_0$ must satisfy the disjointness condition~\eqref{eqn:PGoldRule}, which we can check thanks to Algorithm~\ref{alg:disjoint}. We will demonstrate in the rest of this article that this mode of representation lends itself to extending Makdisi's algorithms to~$\Px$.
 
\begin{rk}\label{rk:kpadic}
We showed in~\cite{Hensel} how Makdisi's algorithms can be extended to the case where~$K$ is a local ring over which the curve~$C$ has good reduction. The same goes for the algorithms for~$\Px$ described in this article, provided that in condition~\eqref{eqn:PGoldRule}, we understand the word ``disjoint'' as meaning ``disjoint over the residue field of~$K$''. In practice, this means that in this situation, the linear algebra calculations in Algorithm~\ref{alg:disjoint} must be performed over the residue field of~$K$; in particular, the rank at line~\ref{alg:disjoint:rank} is then still well-defined.
\end{rk}

\subsection{Algorithmic comparison of points on~$\Px_{x,y}$}

As explained in Section~\ref{sect:compare}, comparing points on~$\Px$ is a nontrivial task, which is further complicated by the non-uniqueness of the representation of points of~$J$ in Makdisi's framework. In order to present an algorithm for this task, we begin by translating Lemma~\ref{lem:Poincare_compare_0} in terms of the symbols~$[D,E]$ introduced in~\eqref{eqn:[DE]}:
\begin{lem}\label{lem:Poincare_compare}
 Let~$D_1, D_2, E_1, E_2 \in \EdC$ be effective divisors of degree~$\delta$. Assume~\eqref{eqn:PGoldRule}, so that the sections~$[D_1,E_1]$ and~$[D_2,E_2]$ make sense. 
 Suppose we have nonzero~$d_1, d_2, e_1, e_2 \in V_2$ such that
 \[D_1+(d_1)_2 = D_2+(d_2)_2 \quad \text{and} \quad E_1+(e_1)_2 = E_2+(e_2)_2, \]
 so that~$D_1 \sim D_2$ and~$E_1 \sim E_2$. If either of the expressions
 \[ \frac{\left(\frac{d_1}{d_2}\right)(E_2) \cdot \left(\frac{e_1}{e_2}\right)(D_1)}{\left(\frac{d_1}{d_2}\right)(E_0) \cdot \left(\frac{e_1}{e_2}\right)(D_0)}
 \qquad \text{or} \qquad
 \frac{\left(\frac{d_1}{d_2}\right)(E_1) \cdot \left(\frac{e_1}{e_2}\right)(D_2)}{\left(\frac{d_1}{d_2}\right)(E_0) \cdot \left(\frac{e_1}{e_2}\right)(D_0)} \]
 does not involve any term of the form~$0/0$, then its value agrees with the unique~$\lambda \in \Kx$ such that~$[D_2,E_2] = \lambda \cdot [D_1,E_1]$.
\end{lem}

\begin{proof}
This follows from Lemma~\ref{lem:Poincare_compare_0}, since~$D_2-D_0 = D_1-D_0 + (d_1/d_2)$ and~$E_2-E_0 = E_1-E_0 +(e_1/e_2)$. 
\end{proof}

In order to evaluate~$\lambda$ algorithmically in Lemma~\ref{lem:Poincare_compare}, we still need to evaluate terms of the form~$\left(\frac{s_1}{s_2}\right)(D)$, where~$s_1$,~$s_2 \in V_2$ and~$D \in \EdC$. This is actually a nontrivial task, since~$D$ will only be given to us as the matrix~$M_D$ representing the subspace~$V_2(-D) \subset V_2$, which only describes~$D$ indirectly, and also because~$s_1$ and~$s_2$ will only be known to us through the column vectors~$c_{s_1}, c_{s_2} \in K^{n_Z}$, which only tell us the values of~$s_1$ and of~$s_2$ at the fixed points~$Z_i$ which have no reason to have any connection with the support of~$D$ where we would like to actually evaluate~$s_1$ and~$s_2$. We therefore present an algorithm achieving this, and on which most of the machinery presented in this article relies.


\begin{algorithm}[H]
\KwIn{An integer~$n \le 3$, two column vectors~$c_{s_1},c_{s_2} \in K^{n_Z}$ representing~$s_1,s_2 \in V_n$, a matrix~$M_D$ encoding~$D \in \EdC$.}
\KwOut{$\left(\frac{s_1}{s_2}\right)(D) \in \Kx$ if~$(s_1)_n$ and~$(s_2)_n$ are both disjoint from~$D$, else \FAIL.}
$M_{V_{n+2}(-D)} \la \DivAdd(V_n,M_D)$\tcp*{Encodes~$V_{n+2}(-D)$,~$n+2 \le 5$.}
$M_S \la$ matrix whose columns represent a basis~$B$ of a linear supplement~$V_2 = V_2(-D) \oplus S$ of~$V_2(-D)$ in~$V_2$\;
$M_{S'} \la$ matrix whose columns represent a basis~$B'$ of a linear supplement~$V_{n+2} = V_{n+2}(-D) \oplus S'$ of~$V_{n+2}(-D)$ in~$V_{n+2}$\;
\For{$i=1$ \KwTo~$2$}{
\tcp{With respect to the bases~$B$ and~$B'$:}
$\delta_i \la \det \left( F_i : S \hookrightarrow V_2 \overset{\odot c_{s_i}}{\longrightarrow} V_{n+2} \twoheadrightarrow S' \right)$\;
\If{$\delta_i = 0$}{\FAIL\tcp{$s_i(D)=0$}}
}
\Return~$\delta_1 / \delta_2$\;
\caption{Evaluating a ratio of sections at a divisor}
\label{alg:norm}
\end{algorithm}


\begin{proof}
This is actually a generalisation of~\cite[Algorithm 5]{Hensel}, to which we refer the reader for a more detailed presentation and proof.
The idea is that since~$V_2(-D) = \Set{s \in V_2}{s_{\vert D}=0}$, the skyscraper sheaf~$\O_D$ may be identified with~$V_2/V_2(-D) \simeq S$, and similarly with~$S'$. Therefore, the maps~$F_i$ represent multiplication by~$s_i$ under these identifications. The~$\delta_i$ thus correspond to~$s_i(D)$, up to an ambiguity due to the choice of the bases~$B$ and~$B'$; but this ambiguity conveniently cancels out in the ratio~$\delta_1/\delta_2$.
\end{proof}

As a first application of Algorithm~\ref{alg:norm}, we can finally show how to compare points on~$\Px$. For this, we will suppose that Makdisi's equality test Algorithm~\ref{alg:MakEq} has been slightly altered so that, when this algorithm returns a positive answer, it actually returns a pair~$(c_u,c_{u'})$, where~$u \in V_2$ is computed at line~\ref{alg:MakEq:u} and~$u' \in V_2$ is a nonzero element of the one-dimensional space~$W$ computed at line~\ref{alg:MakEq:W}, so that~$u'$ is unique up to scaling. The point of this modification is that we then have~$D'+(u)_2 = D+D'+E = D+(u')_2$ in the notations of Algorithm~\ref{alg:MakEq}; in other words, the pair~$(u,u')$ yields an explicit certificate of linear equivalence between~$D$ and~$D'$, and is therefore exactly what we need in order to apply the formula of Lemma~\ref{lem:Poincare_compare}.

\begin{algorithm}[H]
\KwIn{Two points~$\pi_1 = (M_{D_1},M_{E_1},\lambda_1)$ and~$\pi_2 = (M_{D_2},M_{E_2},\lambda_2)$ of~$\Px$.}
\KwOut{The unique~$\lambda \in \Kx$ such that~$\pi_2 = \lambda \cdot \pi_1$ if~$\pi_1$ and~$\pi_2$ lie in the same stalk of~$\Px$;~$0$ else.}
\Repeat{at line~\ref{alg:PoincareCompare:Norms}, all four calls to Algorithm~\ref{alg:norm} succeed}{
\tcp{Using Algorithm~\ref{alg:MakEq}:}
\eIf{$D_1 \sim D_2$}
{$c_{d_1},c_{d_2} \la$ \begin{tabular}c column vectors encoding sections~$ d_1, d_2 \in V_2$ \\ such that~$D_1+(d_1)_2 = D_2+(d_2)_2$\; \end{tabular}}
{\Return 0\;}
\tcp{Using Algorithm~\ref{alg:MakEq} again:}
\eIf{$E_1 \sim E_2$}
{$c_{e_1},c_{e_2} \la$ \begin{tabular}c column vectors encoding sections~$ e_1, e_2 \in V_2$ \\ such that~$E_1+(e_1)_2 = E_2+(e_2)_2$\; \end{tabular}}
{\Return 0\;}
\tcp{Using Algorithm~\ref{alg:norm} (with~$n=2$) four times:}
$\displaystyle \lambda \la \left(\frac{d_1}{d_2}\right)(E_2) \cdot \left(\frac{d_2}{d_1}\right)(E_0)\cdot \left(\frac{e_1}{e_2}\right)(D_1) \cdot \left(\frac{e_2}{e_1}\right)(D_0)$\;\label{alg:PoincareCompare} \label{alg:PoincareCompare:Norms}
}
\Return~$\displaystyle \lambda \frac{\lambda_2}{\lambda_1}$\;
\caption{Comparison of points of~$\Px$}
\label{alg:PoincareCompare}
\end{algorithm}

\begin{proof}
As we explained, this follows from Lemma~\ref{lem:Poincare_compare}. The point of the \textbf{repeat\dots until} loop is that, in Algorithm~\ref{alg:MakEq}, the section~$u$ is chosen at random at line~\ref{alg:MakEq:u}, which makes the pair~$(u,u')$ also random; therefore, we cannot be always certain that Lemma~\ref{lem:Poincare_compare} will apply, albeit we expect it will with high probability for the reason explained in Remark~\ref{rk:EnoughPtsDisjoint}.
\end{proof}

\begin{rk}\label{rk:compare_need_M0N0}
At line~\ref{alg:PoincareCompare:Norms} of Algorithm~\ref{alg:PoincareCompare}, in order to use Algorithm~\ref{alg:norm} to compute the terms~$\left(\frac{d_1}{d_2}\right)(E_0)$ and~$\left(\frac{e_1}{e_2}\right)(D_0)$, we need matrices~$M_{D_0}$ and~$M_{E_0}$ representing respectively the subspaces~$V_2(-D_0) = d_0 \cdot V_1$ and~$V_2(-E_0) = e_0 \cdot V_1$ of~$V_2$, which is one of the reasons why we precomputed these matrices in Section~\ref{sect:repPx}.
\end{rk}

\begin{rk}
Lemma~\ref{lem:Poincare_compare} gives~$\lambda =  \frac{\left(\frac{d_1}{d_2}\right)(E_2) \cdot \left(\frac{e_1}{e_2}\right)(D_1)}{\left(\frac{d_1}{d_2}\right)(E_0) \cdot \left(\frac{e_1}{e_2}\right)(D_0)}$; the formula at line~\ref{alg:PoincareCompare:Norms}, which is of course equivalent, avoids divisions in~$K^\times$, since those are usually slower than multiplications.
\end{rk}

\begin{rk}
The description of~$\Px$ yields canonical trivialisations~$\Px_{0,y} \simeq \O_{J,y} \simeq \Kx$ and~$\Px_{x,0} \simeq \O_{J,x} \simeq \Kx$ which are, in particular, compatible with each other. We can therefore canonically and consistently assign a scalar~$\lambda \in \Kx$ to~$[D,E]$ whenever at least one of~$D, E$ represents~$0 \in J$ (i.e. agrees with~$\L$ in~$\Pic^\delta(C)$). This~$\lambda$ can be computed by an algorithm similar to Algorithm~\ref{alg:PoincareCompare}. Indeed, suppose for example that~$D$ represents~$0 \in J$, and let~$0 \neq d$ be an element of the one-dimensional space~$V_1(-D)$ computed at line~\ref{alg:MakZero:W} of Algorithm~\ref{alg:MakZero}, so that~$d$ is unique up to scaling and satisfies~$(d)_1 = D$. Applying Lemma~\ref{lem:Poincare_compare}, we find that
\[ \lambda = \left(\frac{d}{d_0}\right)(E) \cdot \left(\frac{d_0}{d}\right)(E_0), \]
which we can compute by two calls to Algorithm~\ref{alg:norm} (with~$n=1$ this time) which, in this case, are guaranteed to always succeed by~\eqref{eqn:PGoldRule}.

Similarly, if~$E$ represents~$0 \in J$, we find
\[ \lambda = \left(\frac{e}{e_0}\right)(D) \cdot \left(\frac{e_0}{e}\right)(D_0) \quad \text{where } (e)_1=E, \]
which can be computed in the same way.
\end{rk}

\subsection{The canonical involution of~$\Px$}

The canonical ``swap'' involution
\[ \setmap{J \times J}{J \times J}{(x,y)}{(y,x)} \]
extends to a canonical involution
\[ \tau : \setmap{\Px}{\Px}{\lambda \cdot (D,E)}{\lambda \cdot (E,D)}. \]
Indeed, Lemma~\ref{lem:Poincare_compare_0} confirms that~$\tau$ is well-defined, in the sense that it does not depend on the choice of~$\lambda \in \Kx$ and~$D,E \in \Div^0(C)$ to represent points of~$\Px$; see also~\cite[Section 6.5]{EdiLido}.

In order to design an algorithm for~$\tau$, we need to translate its definition in terms of the symbols~$[D,E]$ defined in~\eqref{eqn:[DE]}. For this, we will require the following generalisation of Weil reciprocity:

\begin{lem}\label{lem:WeilSections}
Let~$X$ be a ``nice'' curve,~$\M$ a line bundle on~$X$, and~$s,t,u$ three meromorphic sections of~$\M$ whose divisors~$S \defeq (s)_\M$,~$T \defeq (t)_\M$, and~$U \defeq (u)_\M$ are pairwise disjoint, so that the scalar
\[ \lambda \defeq \left(\frac{s}{t}\right)(U) \cdot \left(\frac{t}{u}\right)(S) \cdot \left(\frac{u}{s}\right)(T) \]
is well-defined. Then~$\lambda$ does not depend on~$s,t,u$, and in fact,~$\lambda = (-1)^{\deg \M}$.
\end{lem}

\begin{proof}
Write~$S = \sum_{P \in X} n_{S,P} P$, and similarly for~$T$ and~$U$. Recall the version of Weil reciprocity for non-disjoint divisors: Whenever~$w$ and~$z$ are nonzero rational functions on~$X$, we have~$\prod_{P \in X} \left\{w,z\right\}_P = 1$, where
\[ \left\{w,z\right\}_P = (-1)^{\ord_P w \ord_P z} \frac{w^{\ord_P z}}{z^{\ord_P w}}(P); \]
in particular,~$\left\{w,z\right\}_P = 1$ for~$P$ away from the support of~$(w)$ and of~$(z)$.

Apply this to~$w=t/u$ and~$z=s/u$. As~$S$,~$T$, and~$U$ are pairwise disjoint, we find that
\begin{itemize}
\item if~$P \in \Supp S$, then~$\ord_P w = 0$ and~$\ord_P z = n_{S,P}$, whence 
\[ \left\{w,z\right\}_P = (-1)^0 \frac{w^{n_{S,P}}(P)}1 = \left( \frac{t}{u} (P)\right)^{n_{S,P}}, \]
\item if~$P \in \Supp T$, then~$\ord_P w = n_{T,P}$ and~$\ord_P z = 0$, whence 
\[ \left\{w,z\right\}_P = (-1)^0 \frac1{z^{n_{T,P}}(P)} = \left( \frac{u}{s} (P)\right)^{n_{T,P}}, \]
\item if~$P \in \Supp U$, then~$\ord_P w  = \ord_P z = -n_{U,P}$, whence 
\[ \left\{w,z\right\}_P = (-1)^{(-n_{U,P})^2} \frac{w^{-n_{U,P}}(P)}{z^{-n_{U,P}}(P)} = (-1)^{n_{U,P}} \left( \frac{s}{t} (P)\right)^{n_{U,P}}, \]
\item and~$\left\{w,z\right\}_P =1$ else.
\end{itemize}
Therefore Weil reciprocity tells us that
\begin{align*}
1 &= \prod_{P \in X} \left\{w,z\right\}_P \\
&= \prod_{P \in \Supp S} \left\{w,z\right\}_P \prod_{P \in \Supp T} \left\{w,z\right\}_P \prod_{P \in \Supp U} \left\{w,z\right\}_P \\
&= \left(\frac{t}{u}\right)(S) \cdot \left(\frac{u}{s}\right)(T) \cdot (-1)^{\deg U} \left(\frac{s}{t}\right)(U).
\end{align*}

The author thanks Davide Lombardo for this proof.
\end{proof}

\begin{lem}\label{lem:swap:formula}
Let~$\pi = [D,E] \in \Px$, where~$D,E \in \EdC$ are such that~$D$,~$E$,~$D_0$, and~$E_0$ are pairwise disjoint. Then~$[E,D]$ is a well-defined point of~$\Px$, and
\[ \tau(\pi) = (-1)^\delta \left(\frac{d_0}{e_0}\right)(D-E) \cdot [E,D]. \]
\end{lem}

\begin{proof}
By definition,
\[ \tau(\pi) \defeq \tau\big( (D-D_0,E-E_0) \big) \defeq (E-E_0,D-D_0), \]
which we would like to compare to~$[E,D] \defeq (E-D_0,D-E_0)$.

We must resist the temptation of writing
\[ (E-E_0,D-D_0) = (E-D_0+D_0-E_0,D-E_0+E_0-D_0) \]
and applying Lemma~\ref{lem:Poincare_compare_0}, as this leads to terms of the form~$0/0$ (and indeed the disjointness hypothesis of Lemma~\ref{lem:Poincare_compare_0} is not satisfied). Therefore, we instead introduce two nonzero sections~$s,t \in V_1$, which (possibly after extending scalars) we assume are generic enough that the divisors~$D$,~$E$,~$D_0$,~$E_0$,~$S \defeq (s)_1$, and~$T \defeq (t)_1$ are pairwise disjoint, and calculate
\[ (E-E_0,D-D_0) = (E-S+(s/e_0),D-T+(t/d_0)) = \lambda_1 \cdot (E-S,D-T) \]
where, by Lemma~\ref{lem:Poincare_compare_0},
\[ \lambda_1 =  \left(\frac{s}{e_0}\right)(D-D_0) \cdot  \left(\frac{t}{d_0}\right)(E-S). \]
Applying Lemma~\ref{lem:Poincare_compare_0} once more, we find
\[ (E-S,D-T) = (E-D_0+(d_0/s),D-E_0+(e_0/t)) = \lambda_2 \cdot (E-D_0, D-E_0) , \]
where
\[ \lambda_2 = \left(\frac{d_0}{s}\right)(D-T) \cdot \left(\frac{e_0}{t}\right)(E-D_0). \]
Thus~$\tau([D,E]) = \lambda \cdot (E-D_0,D-E_0) \defeq \lambda \cdot [E,D]$, where
\begin{align*}
\lambda & \defeq \lambda_1 \cdot \lambda_2 \\
&= \left(\frac{s}{e_0}\right)(D-D_0) \cdot  \left(\frac{t}{d_0}\right)(E-S) \cdot \left(\frac{d_0}{s}\right)(D-T) \cdot \left(\frac{e_0}{t}\right)(E-D_0) \\
&= \left(\frac{d_0}{e_0}\right)(D) \cdot \left(\frac{e_0}{d_0}\right)(E) \cdot \left(\frac{t}{s}\right)(D_0) \cdot \left(\frac{d_0}{t}\right)(S) \cdot \left(\frac{s}{d_0}\right)(T),
\end{align*}
and we conclude by Lemma~\ref{lem:WeilSections}.
\end{proof}

Lemma~\ref{lem:swap:formula} assumes that~$D$,~$E$,~$D_0$, and~$E_0$ are pairwise disjoint, which is a stricter requirement than our rule~\eqref{eqn:PGoldRule}. In order to reduce to this case, we need an effective version of the moving lemma; this is achieved by Algorithm~\ref{alg:MovingLemma:Px} below, whose presentation we defer to the next section because it relies on Makdisi's algorithms for the group law in~$J$ which we have not introduced yet. We can thus implement the canonical involution~$\tau$ of~$\Px$ as follows:

\begin{algorithm}[H]
\KwIn{A point~$\pi = (M_D,M_E,\lambda)$ of~$\Px$.}
\KwOut{The image of~$\pi$ under the involution~$\tau$.}
Attempt to use Algorithm~\ref{alg:norm} (with~$n=1$) twice to evaluate~$\displaystyle \mu \la \left(\frac{d_0}{e_0}\right)(D) \cdot \left(\frac{e_0}{d_0}\right)(E)$\; \label{alg:Sawp:start}
\If{Algorithm~\ref{alg:norm} succeeded}{\Return~$(M_E,M_D,(-1)^\delta \lambda \mu)$\;}
\tcp{Using Algorithm~\ref{alg:MovingLemma:Px}:}
$(M_{D},M_{E},\lambda) \la$ another representation of~$\pi$\;
Go back to line~\ref{alg:Sawp:start}\;
\caption{Canonical involution on~$\Px$}
\label{alg:Swap}
\end{algorithm}

\begin{rk}
In view of Remark~\ref{rk:EnoughPtsDisjoint}, disjointness is very likely, so we expect Algorithm~\ref{alg:Swap} to terminate after very few iterations. Also note that instead of relying on Algorithm~\ref{alg:disjoint} to enforce the disjointness hypothesis of Lemma~\ref{lem:swap:formula}, Algorithm~\ref{alg:Swap} simply watches whether Algorithm~\ref{alg:norm} succeeds. While crude, this technique is quite efficient in practice, so we often use it in the other algorithms presented in this article. 
\end{rk}

\section{Algorithms for the partial group laws}\label{sect:Laws}

We will now present algorithms for the partial group laws of~$\Px$.

\begin{rk}\label{rk:OnlyLeft}
We will actually restrict ourselves to algorithms for the left partial group law; indeed, there is absolutely no need to use the ``swap'' Algorithm~\ref{alg:Swap} to obtain the corresponding algorithms for the right partial group law, since this can be achieved much more simply by swapping the~$D$'s and the~$E$'s and~$d_0$ with~$e_0$ in the arguments of these algorithms.
\end{rk}

\subsection{Review of Makdisi's algorithms for the group law of~$J$}\label{sect:MakLaw}

Let us briefly review Makdisi's algorithms for the group law of~$J$, before we extend them to the partial group laws of~$\Px$ in Section~\ref{sect:alg_gp_Px} below.

Similarly to the chord-tangent process defining the group law on an elliptic curve, Makdisi actually defines two separate algorithms, one (Algorithm~\ref{alg:MakAddFlip} below) to compute the so-called AddFlip:~$x,y \mapsto -(x+y)$, and one (Algorithm~\ref{alg:MakNeg} below) for the Negation~$x \mapsto -x$. It is clear that these algorithms can be combined to compute not only the addition~$x,y \mapsto x+y$, but also the subtraction~$x,y \mapsto x-y$.

Recall that~$M_D$ denotes a matrix representing the subspace~$V_2(-D) \subset V_2$ encoding~$D \in \EdC$.

\begin{algorithm}[H]
\KwIn{Matrices~$M_{D_1}$ and~$M_{D_2}$ representing the points~$x_1 = [\Ld(D_1)]$ and~$x_2 = [\Ld(D_2)] \in J(K)$, where~$D_1, D_2 \in \Eff^{\delta}(C)$.}
\KwOut{A matrix~$M_{D_3}$ representing the point~$x_3 = [\Ld(D_3)]$ such that~$x_1+x_2+x_3=0$.}
$M_{V_4(-D_1-D_2)} \la \DivAdd(M_{D_1},M_{D_2})$\;
$M_{V_3(-D_1-D_2)} \la \DivSub(M_{V_4(-D_1-D_2)},M_{V_1})$\;
$0 \neq c_s \la$ a random vector in the column span of~$M_{V_3(-D_1-D_2)}$\; \label{alg:MakAddFlip:s}
\tcp{Represents~$0 \neq s \in V_3(-D_1-D_2)$}
\tcp{$\leadsto (s)_3 = D_1+D_2+D_3$ for some~$D_3 \in \EdC$}
$M_{s \cdot V_2} \la \FnMul(c_s,V_2)$\tcp*{Represents~$V_5(-D_1-D_2-D_3)$}
$M_{D_3} \la \DivSub(M_{s \cdot V_2}, M_{V_3(-D_1-D_2)})$\;
\Return~$M_{D_3}$\;
\caption{Makdisi's AddFlip~$(x_1,x_2) \mapsto -(x_1+x_2)$ in~$J$}
\label{alg:MakAddFlip}
\end{algorithm}

\begin{algorithm}[H]
\KwIn{A matrix~$M_{D}$ representing the point~$x = [\Ld(D)] \in J(K)$, where~$D \in \Eff^{\delta}(C)$.}
\KwOut{A matrix~$M_{D'}$ representing the point~$x' = [\Ld(D')]$ such that~$x+x'=0$.}
$0 \neq c_s \la$ a random vector in the column span of~$M_{D}$\; \label{alg:MakNeg:s}
\tcp{Represents~$0 \neq s \in V_2(-D)$}
\tcp{$\leadsto (s)_2 = D+D'$ for some~$D' \in \EdC$}
$M_{s \cdot V_2} \la \FnMul(c_s,V_2)$\tcp*{Represents~$V_4(-D-D')$}
$M_{D'} \la \DivSub(M_{s \cdot V_2}, M_{D})$\;
\Return~$M_{D'}$\;
\caption{Makdisi's Negation~$x \mapsto -x$ in~$J$}
\label{alg:MakNeg}
\end{algorithm}

\subsection{Effective moving lemmas}

In order to keep track of linear equivalences, we will assume from now on that Algorithm~\ref{alg:MakAddFlip} has been altered so as to also return the random vector~$c_s$ that it generates at line~\ref{alg:MakAddFlip:s}, and similarly that Algorithm~\ref{alg:MakNeg} now also returns the random vector~$c_s$ generated at line~\ref{alg:MakNeg:s}.

For example, by applying twice the Negation Algorithm~\ref{alg:MakNeg} thus modified, we obtain an effective version of the moving lemma for~$J$, which is reminiscent of the equality test Algorithm~\ref{alg:MakEq} but serves a different purpose:

\begin{algorithm}[H]
\KwIn{A matrix~$M_{D}$ representing the point~$x = [\Ld(D)] \in J(K)$, where~$D \in \Eff^{\delta}(C)$.}
\KwOut{A matrix~$M_{D'}$ representing the same point~$x$, and a pair of vectors~$c_{s}$,~$c_{s'}$, where~$s,s' \in V_2$ satisfy~$D+(s)_2 = D'+(s')_2$.}
$M_{D''}, c_{s'} \la \Negation(M_D)$\tcp*{$(s')_2 = D+D''$}
$M_{D'}, c_{s} \la \Negation(M_{D''})$\tcp*{$(s)_2 = D''+D'$}
\Return~$M_{D'},c_s,c_{s'}$\;
\caption{Moving lemma for~$J$}
\label{alg:MovingLemma:J}
\end{algorithm}

By combining Algorithm~\ref{alg:MovingLemma:J} with Lemma~\ref{lem:Poincare_compare}, we obtain an effective moving lemma for~$\Px$.

\begin{algorithm}[H]
\KwIn{A point~$\pi = (M_D,M_E,\lambda)$ of~$\Px$.}
\KwOut{Another representation~$(M_{D'},M_{E'},\lambda')$ of~$\pi$.}
\Repeat{Algorithm~\ref{alg:norm} succeeds in evaluating~$\mu$}
{
\tcp{Using Algorithm~\ref{alg:MovingLemma:J} twice:}
Find~$M_{D'},c_d,c_{d'}$ and~$M_{E'},c_e,c_{e'}$ with~$d,d',e,e' \in V_2$ such that~$D+(d)_2 = D'+(d')_2$ and~$E+(e)_2 = E'+(e')_2$\;
\tcp{Using Algorithm~\ref{alg:norm} (with~$n=2$) four times:}
$\displaystyle \mu \la \left(\frac{d'}{d}\right)(E) \cdot \left(\frac{d}{d'}\right)(E_0)\cdot \left(\frac{e'}{e}\right)(D') \cdot \left(\frac{e}{e'}\right)(D_0)$\;
}
\Return~$(M_{D'},M_{E'}, \lambda \cdot \mu)$\;
\caption{Moving lemma for~$\Px$}
\label{alg:MovingLemma:Px}
\end{algorithm}

\begin{rk}
We will only require this moving lemma algorithm for~$\Px$ for the implementation of the canonical involution of~$\Px$ (Algorithm~\ref{alg:Swap} above).
\end{rk}

\subsection{Alignment procedures}\label{sect:align}

Before we present our algorithms for the partial group laws in~$\Px$, we wish to address the slight technical difficulty highlighted by the following example. Suppose we are given two points~$\pi \in \Px_{x,y}$ and~$\pi' \in \Px_{x',y'}$ of~$\Px$ in the format~\eqref{eqn:codepi}, that is to say as~$\pi = \lambda \cdot [D,E]$ and~$\pi'=\lambda \cdot [D',E']$, where~$x,y,x',y' \in J$ are the points represented by~$D,E,D',E' \in \EdC$ respectively and~$\lambda, \lambda' \in \Kx$; and suppose for example that we wish to compute the left partial group law~$\pi \Ltimes \pi'$. This obviously assumes that~$y=y'$; however, this does not imply that~$E=E'$ as divisors, but merely that~$E \sim E'$ are linearly equivalent.

\begin{de}
We will say that the calculation~$[D,E] \Ltimes [D',E']$ is \emph{aligned} if~$E=E'$ as divisors, and \emph{misaligned} else.
\end{de}

In the misaligned case (which is, of course, the general one), the algorithmic calculation of~$\pi \Ltimes \pi'$ would then proceed in two steps:
\begin{enumerate}
\item An \emph{alignment step}, where we use some variant of Algorithm~\ref{alg:PoincareCompare} to find~$\mu \in \Kx$ such that~$[D',E'] = \mu \cdot [D',E]$,
\item and the crux of the partial group law calculation, where we find~$D'' \in \EdC$ representing~$x+x'$ and~$\lambda'' \in \Kx$ such that
\[ [D,E] \Ltimes [D',E] = \lambda'' \cdot [D'',E], \]
an \emph{aligned} calculation.
\end{enumerate}
The final result would of course be~$\pi \Ltimes \pi' = \lambda \lambda' \mu \lambda'' \cdot [D'',E]$.

From this example, it is clear that the necessity to start with the alignment step, although trivial, would complicate the code for the partial group laws and obfuscate its interesting part, which is the group law calculation in the aligned case. Therefore, before we present algorithms to compute the partial group laws, in this section, we will present algorithms performing the alignment step. Thus, when we present algorithms for the partial group laws in the following sections, we will be able to assume without loss of generality that we are in the aligned case, which will considerably simplify the exposition of these algorithms. For similar reasons, in practice, we would recommend that alignment procedures be implemented as separate functions from the partial group laws, so as to improve the readability and modularity of the code. Another advantage of this approach is that in situations where we want to perform many partial group operations on many points, e.g. to compute a left linear combination
\[ \LTimes_i \pi_i^{\Ltimes n_i} \]
for some points~$\pi_i \in \Px_{x_i,y}$ and integers~$n_i \in \Z$, this leads to the alignment being performed once and for all at the beginning of the calculation, rather than having to perform a realignment at every step of the calculation. 

We will actually present two alignment procedures. The first one, called \emph{weak alignment}, can only manage two points of~$\Px$, and may fail (with low probability), whereas the second one, called \emph{strong alignment}, is guaranteed to eventually succeed and can handle arbitrarily many points, but is slower.

As explained in Remark~\ref{rk:OnlyLeft}, we limit ourselves to describing these procedures in the setting of \emph{left} partial group operations in~$\Px$, where our points have the same~$y$-coordinate, which are thus represented by divisors~$E_i$ which are all equivalent to each other, but not necessarily equal as divisors.

\begin{algorithm}[H]
\KwIn{Matrices~$M_{D}$,~$M_{E_1}$, and~$M_{E_2}$ representing the points~$x = [\Ld(D)]$,~$y_1 = [\Ld(E_1)]$, and~$y_2 = [\Ld(E_2)] \in J(K)$ where~$D, E_1, E_2 \in \Eff^{\delta}(C)$ and such that~$y_1=y_2$.}
\KwOut{The unique scalar~$\lambda \in \Kx$ such that~$[D,E_2] = \lambda \cdot [D,E_1]$, or \FAIL.}
\tcp{Using Algorithm~\ref{alg:MakEq}:}
Find~$c_{e_1},c_{e_2}$ where~$e_1,e_2 \in V_2$ are such that~$E_1+(e_1)_2 = E_2+(e_2)_2$\;
\tcp{Using Algorithm~\ref{alg:norm} (with~$n=2$) twice (may fail):}
$\displaystyle \lambda \la \left(\frac{e_1}{e_2}\right)(D) \cdot \left(\frac{e_2}{e_1}\right)(D_0)$\;
\caption{Attempt to align two points of~$\Px$}
\label{alg:WeakAlign}
\end{algorithm}

\begin{algorithm}[H]
\KwIn{Matrices~$M_{D_1},M_{E_1}, \cdots, M_{D_m}, M_{E_m}$ representing~$x_1 = [\Ld(D_1)], y_1 = [\Ld(E_1)], \cdots, x_m = [\Ld(D_m)], y_m = [\Ld(E_m)]$, where~$D_1,E_1, \cdots, D_m, E_m \in \Eff^{\delta}(C)$ and such that~$y_1=\cdots=y_m$.}
\KwOut{A matrix~$M_{E}$ representing the point~$y = [\Ld(E)]$ such that~$y=y_1=\cdots=y_m$, and the unique scalars~$\lambda_1,\cdots,\lambda_m \in \Kx$ such that~$[D_i,E_i] = \lambda_i \cdot [D_i,E]$ for all~$1 \le i \le m$.}
\Repeat{all calls to Algorithm~\ref{alg:norm} succeed at line~\ref{alg:StrongAlign:norms}}
{
\tcp{Pick random~$E \sim E_1$:}
Use Algorithm~\ref{alg:MovingLemma:J} to find~$M_{E},c_{s_1},c_{t_1}$ with~$s_1,t_1 \in V_2$ such that~$E+(s_1)_2 = E_1+(t_1)_2$\;
\tcp{Align the others~$E_i$ on~$E$:}
\For{$i=2$ \KwTo~$m$}
{
\tcp{Using Algorithm~\ref{alg:MakEq}:}
$(s_i,t_i) \la$ elements of~$V_2$ such that~$E+(s_i)_2 = E_i + (t_i)_2$\;
}
\For{$i=1$ \KwTo~$n$}
{
\tcp{Using Algorithm~\ref{alg:norm} (with~$n=2$) twice:}
$\displaystyle \lambda_i \la \left( \frac{s_i}{t_i} \right)(D_i) \cdot \left( \frac{t_i}{s_i} \right)(D_0)$\;\label{alg:StrongAlign:norms}
}
}
\Return~$(M_E,\lambda_1,\cdots,\lambda_m)$\;
\caption{Align~$m$ points of~$\Px$}
\label{alg:StrongAlign}
\end{algorithm}

Both Algorithms~\ref{alg:WeakAlign} and~\ref{alg:StrongAlign} are straightforward consequences of Lemma~\ref{lem:Poincare_compare}. The for loops in Algorithm~\ref{alg:StrongAlign} may, of course, be executed in parallel.

\subsection{Algorithms for aligned points on~$\Px$} \label{sect:alg_gp_Px}

We now show how to extend Algorithms~\ref{alg:MakAddFlip} and~\ref{alg:MakNeg} to compute the partial group laws of~$\Px$. Since we write these laws multiplicatively whereas the law of~$J$ is additive, and since we do not wish to concede to Makdisi the primacy on debatable terminology, we call these extensions MulFlip and Inversion, respectively. As explained in Remark~\ref{rk:OnlyLeft}, we only present our algorithms for the left partial group law~$\Ltimes$, since algorithms for the right partial group law~$\Rtimes$ can be obtained simply by swapping~$d_0$ with~$e_0$ and the~$D$'s with the~$E$'s. Furthermore, we can limit ourselves to aligned operations, thanks to the procedures described in Section~\ref{sect:align}.

In order to keep track of the linear equivalences, we suppose that Algorithm~\ref{alg:MakAddFlip} (respectively,~\ref{alg:MakNeg}) has been modified so as to also return the column vector~$c_s$ computed at line~\ref{alg:MakAddFlip:s} (respectively, at line~\ref{alg:MakNeg:s}).

\begin{algorithm}[H]
\KwIn{Matrices~$M_{D_1}$,~$M_{D_2}$, and~$M_{E}$ representing the points~$x_1 = [\Ld(D_1)]$,~$x_2 = [\Ld(D_2)]$, and~$y = [\Ld(E)] \in J(K)$ where~$D_1, D_2, E \in \Eff^{\delta}(C)$.}
\KwOut{A matrix~$M_{D_3}$ representing the point~$x_3 = [\Ld(D_3)]$ such that~$x_1+x_2+x_3=0$, and a scalar~$\lambda \in \Kx$ such that~$\left( [D_1,E] \Ltimes [D_2,E] \right)^{\Ltimes -1} = \lambda \cdot [D_3,E]$.}
\Repeat{Algorithm~\ref{alg:disjoint} verifies that~$D_3$ is disjoint from~$E$ and~$E_0$}{
$(M_{D_3},c_s) \la \AddFlip(M_{D_1},M_{D_2})$\tcp*{Using Algorithm~\ref{alg:MakAddFlip}, \\ so~$(s)_3=D_1+D_2+D_3$}
}
\tcp{Using Algorithm~\ref{alg:norm} (with~$n=3$) twice:}
$\displaystyle \lambda \la \left(\frac{s}{d_0^3}\right)(E_0) \cdot \left(\frac{d_0^3}{s}\right)(E)$\;
\Return~$(M_{D_3},\lambda)$\;
\caption{Left MulFlip~$\Px_{x_1,y} \times \Px_{x_2,y} \longrightarrow \Px_{-x_1-x_2,y}$}
\label{alg:PMulFlip}
\end{algorithm}
Thus, given two triples~$(M_{D_1},M_E,\lambda_1)$ and~$(M_{D_1},M_E,\lambda_1)$ encoding respectively the points~$\pi_1 = \lambda_1 \cdot [D_1,E] \in \Px_{x_1,y}$ and~$\pi_2 = \lambda_2 \cdot [D_2,E] \in \Px_{x_2,y}$, we can use Algorithm~\ref{alg:PMulFlip} to compute the representation~$(M_{D_3},M_E,\lambda/\lambda_1 \lambda_2)$ of~$(\pi_1 \Ltimes \pi_2)^{\Ltimes -1}$.

\begin{algorithm}[H]
\KwIn{Matrices~$M_{D}$ and~$M_{E}$ representing the points~$x = [\Ld(D)]$ and~$y = [\Ld(E)] \in J(K)$ where~$D, E \in \Eff^{\delta}(C)$.}
\KwOut{A matrix~$M_{D'}$ representing the point~$x' = [\Ld(D')]$ such that~$x+x'=0$, and a scalar~$\lambda \in \Kx$ such that~$[D,E]^{\Ltimes -1} = \lambda \cdot [D',E]$.}
\Repeat{Algorithm~\ref{alg:disjoint} verifies that~$D'$ is disjoint from~$E$ and~$E_0$}{
$(M_{D'},c_s) \la \Negation(M_{D})$\tcp*{Using Algorithm~\ref{alg:MakNeg}, \\ so~$(s)_2=D+D'$}
}
\tcp{Using Algorithm~\ref{alg:norm} (with~$n=2$) twice:}
$\displaystyle \lambda \la \left(\frac{s}{d_0^2}\right)(E_0) \cdot \left(\frac{d_0^2}{s}\right)(E)$\;
\Return~$(M_{D'},\lambda)$\;
\caption{Left inversion~$\Px_{x,y} \longrightarrow \Px_{-x,y}$}
\label{alg:PInv}
\end{algorithm}
Thus, given~$(M_{D},M_E,\mu)$ encoding the point~$\pi = \mu \cdot [D,E] \in \Px_{x,y}$, we can use Algorithm~\ref{alg:PInv} to compute the representation~$(M_{D'},M_E,\lambda/\mu)$ of~$\pi^{\Ltimes -1}$.

Obviously, the left partial group law can then be computed by using Algorithms~\ref{alg:PMulFlip} and then~\ref{alg:PInv}. However, combining these two algorithms into one allows us to achieve the same result with only one call to Algorithm~\ref{alg:norm}, which is interesting as Algorithm~\ref{alg:norm} can be computationally costly.

\begin{algorithm}[H]
\KwIn{Matrices~$M_{D_1}$,~$M_{D_2}$, and~$M_{E}$ representing the points~$x_1 = [\Ld(D_1)]$,~$x_2 = [\Ld(D_2)]$, and~$y = [\Ld(E)] \in J(K)$ where~$D_1, D_2, E \in \Eff^{\delta}(C)$.}
\KwOut{A matrix~$M_{D_4}$ representing the point~$x_4 = [\Ld(D_4)]$ such that~$x_4=x_1+x_2$, and a scalar~$\lambda \in \Kx$ such that~$[D_1,E] \Ltimes [D_2,E] = \lambda \cdot [D_4,E]$.}
\Repeat{Algorithm~\ref{alg:disjoint} verifies that~$D_3$ is disjoint from~$E$ and~$E_0$}{
$(M_{D_3},c_s) \la \AddFlip(M_{D_1},M_{D_2})$\tcp*{Using Algorithm~\ref{alg:MakAddFlip}, \\ so~$(s)_3=D_1+D_2+D_3$}
}
\Repeat{Algorithm~\ref{alg:disjoint} verifies that~$D_4$ is disjoint from~$E$ and~$E_0$}{
$(M_{D_4},c_t) \la \Negation(M_{D_3})$\tcp*{Using Algorithm~\ref{alg:MakNeg}, \\ so~$(t)_2=D_3+D_4$}
}
\tcp{Using Algorithm~\ref{alg:norm} ($n=3$) twice:}
$\displaystyle \lambda \la \left(\frac{s}{t d_0}\right)(E) \cdot \left(\frac{t d_0}{s}\right)(E_0)$\;
\Return~$(M_{D_4},\lambda)$\;
\caption{Left partial group law~$\Px_{x_1,y} \times \Px_{x_2,y} \longrightarrow \Px_{x_1+x_2,y}$}
\label{alg:Pmul}

\end{algorithm}
\begin{rk}
In Algorithms~\ref{alg:PMulFlip},~\ref{alg:PInv}, and~\ref{alg:Pmul}, the \textbf{repeat \dots until} loops ensure that the output still satisfies our golden rule~\eqref{eqn:PGoldRule}, which also ensures that the calls to Algorithm~\ref{alg:norm} are guaranteed to always succeed.
\end{rk}

\subsection{Fast exponentiation and (bi)linear combinations}

\subsubsection{Fast exponentiation}

Given a point~$\pi \in \Px_{x,y}$ and an integer~$n \in \Z$, it is easy to derive from Algorithms~\ref{alg:PMulFlip} and~\ref{alg:PInv} a \emph{fast left exponentiation} algorithm to compute~$\pi^{\Ltimes n} \in \Px_{nx,y}$ in~$O(\log \vert n \vert)$ steps, thanks to the concept of \emph{AddFlip chain} discussed in~\cite[2.2.1]{Hensel}. This is a better than using Algorithm~\ref{alg:Pmul} and standard addition chains, since Algorithm~\ref{alg:PMulFlip} is faster than Algorithm~\ref{alg:Pmul}. The same goes, of course, for the computation of~$\pi^{\Rtimes n} \in \Px_{x,ny}$.

\subsubsection{Linear combinations}

Suppose now that we are given several points~$\pi_i \in \Px_{x_i,y}$ for varying~$x_i \in J$ but fixed~$y \in J$, as well as integers~$n_i \in \Z$, and that we wish to compute the \emph{left linear combination}
\[ \LTimes_i \pi_i^{\Ltimes n_i} \in \Px_{\sum_i n_i x_i, y}. \]
It is clear that this is achievable with a combination of the algorithms that we have just described, but somewhat tricky to implement in a way that avoids unnecessary calls to Makdisi's algorithms (and, in particular, to Algorithms~\ref{alg:disjoint} and~\ref{alg:norm}). We therefore provide an optimised algorithm for this, limiting ourselves to the aligned case thanks to Algorithm~\ref{alg:StrongAlign}.

\begin{algorithm}[H]
\KwIn{Matrices~$M_{D_1},\cdots, M_{D_s},M_E$ representing the points~$x_1, \cdots, x_s, y \in J(K)$, and integers~$n_1\cdots,n_s \in \Z$.}
\KwOut{A matrix~$M_D$ representing~$\sum_{1 \le i \le s} n_i x_i \in J(K)$, and a scalar~$\lambda \in \Kx$ such that~$\LTimes_{1 \le i \le s} [D_i,E]^{\Ltimes n_i} = \lambda \cdot [D,E]$.}
Find~$(M_{D},\lambda)$ such that~$[D_1,E]^{\Ltimes (-1)^{s-1} TODO check n_1} = \lambda \cdot [\Delta_1,E]$\;
\For{$i=2$ \KwTo s}{
\tcp{Now~$\lambda \cdot [D,E]=\psi_{i-1}$}
$n'_i \la (-1)^{s+1-i} n_i$\; \label{alg:PLLC:n'_i}
\tcp{Using fast exponentiation:}
Find~$(M_{\Delta_i},\mu_i)$ such that~$[D_i,E]^{\Ltimes n'_i} = \mu_i \cdot [\Delta_i,E]$\;
$(M_{D},\mu'_i) \la \MulFlip(M_D,M_{\Delta_i},M_E)$\tcp*{Using Algorithm~\ref{alg:PMulFlip}}
$\displaystyle \lambda \la \frac{\mu'_i}{\mu_i \lambda}$\;
\tcp{Now~$\lambda \cdot [D,E]=\psi_{i}$}
}
\Return~$(M_D,\lambda)$\;
\caption{Left linear combination}
\label{alg:PLLC}
\end{algorithm}

\begin{proof}
For each~$1 \le i \le s$, define 
\[ \phi_i = \LTimes_{j \le i} [D_j,E]^{\Ltimes n_j} \qquad \text{and} \qquad \psi_i = \phi_i^{\Ltimes (-1)^{s-i}}. \]
Then~$\psi_1 = \phi_1^{\Ltimes (-1)^{s-1}} = [D_1,E]^{\Ltimes (-1)^{s-1}}$, and for~$2 \le i \le s$,
\begin{align*}
\psi_i &= \phi_i^{\Ltimes (-1)^{s-i}} = \left( \phi_{i-1} \Ltimes [D_i,E]^{\Ltimes n_i} \right)^{\Ltimes (-1)^{s-i}} \\
&=  \psi_{i-1}^{\Ltimes -1} \Ltimes [D_i,E]^{\Ltimes -(-1)^{s+1-i} n_i} = \MulFlip\big(\psi_{i-1}, [D_i,E]^{\Ltimes n'_i}\big),
\end{align*}
where~$n'_i = (-1)^{s+1-i} n_i$ is computed at line~\ref{alg:PLLC:n'_i}.

We thus see by induction that we have~$\lambda \cdot [D,E] = \psi_{i-1}$ at the beginning of each iteration the \textbf{for} loop, and~$\lambda \cdot [D,E] = \psi_{i}$ at the end of each iteration. Hence, when we exit the loop, we have obtained~$\psi_{s} = \phi_s$, which is what we were after.
\end{proof}

\subsubsection{Bilinear combinations}

Suppose now that we are in the situation where we are given
\begin{itemize}
\item~$r \in \N$ points~$x_1, \cdots, x_r \in J(K)$,
\item as many integers~$m_1,\cdots,m_r \in \Z$,
\item~$s \in \N$ more points~$y_1, \cdots, y_s \in J$,
\item as many extra integers~$n_1,\cdots,n_s \in \Z$,
\item and finally points~$\pi_{i,j} \in \Px$ for~$1 \le i \le r$ and~$1 \le j \le s$.
\end{itemize}
We can then define the \emph{bilinear combination} of the~$\pi_{i,j}$ with coefficients~$m_1,\cdots,m_r$ and~$n_1,\cdots,n_s$, which we suggestively denote by~${}^\trans m \pi n$, as
\[ \LTimes_{1 \le i \le r} \left( \RTimes_{1 \le j \le s} \pi_{i,j}^{\Rtimes n_j}\right)^{\Ltimes m_i} = \RTimes_{1 \le j \le s} \left( \LTimes_{1 \le i \le r} \pi_{i,j}^{\Ltimes m_i}\right)^{\Rtimes n_j} \in \Px_{\sum_i m_i x_i, \sum_j n_j y_j}, \]
the equality between these two formulas being a consequence of the biextension axioms~\cite[(2.5)]{EdiLido}.
In order to compute it algorithmically (using for example the first formula), we proceed as follows. 

\begin{algorithm}[H]
\KwIn{Integers~$m_1,\cdots,m_r,n_1,\cdots,n_s \in \Z$, and points~$\pi_{i,j} = (M_{D_{i,j}},M_{E_{i,j}},\lambda_{i,j})$ for~$1 \le i \le r$ and~$1 \le j \le s$ such that~$\pi_{i,j} \in \Px_{x_i,y_j}$ for some~$x_i, y_j \in J(K)$.}
\KwOut{The bilinear combination~${}^\trans m \pi n$.}
\For{$i=1$ \KwTo~$r$}{
Use Algorithm~\ref{alg:StrongAlign} (rightside variant) to find~$D'_i \in \EdC$ and~$\mu_1, \cdots, \mu_s \in \Kx$ such that~$[D_{i,j},E_{i,j}] = \mu_j \cdot [D'_i,E_{i,j}]$ for all~$1 \le j \le s$\;
Use Algorithm~\ref{alg:PLLC} (rightside variant) to find~$E'_i \in \EdC$ and~$\nu_i \in \Kx$ such that~$\RTimes_{1 \le j \le s} [D'_i,E_{i,j}]^{\Rtimes n_j} = \nu_i \cdot [D'_i,E'_i]$\;
$\nu_i \la \nu_i \prod_{1 \le j \le s} (\lambda_{i,j} \mu_j)^{n_j}$\; 
}
Use Algorithm~\ref{alg:StrongAlign} (leftside variant) to find~$E'' \in \EdC$ and~$\rho_1, \cdots, \rho_r \in \Kx$ such that~$[D'_i,E'_i] = \rho_i \cdot [D'_i,E'']$ for all~$1 \le i \le r$\;
Use Algorithm~\ref{alg:PLLC} (leftside variant) to find~$D'' \in \EdC$ and~$\sigma \in \Kx$ such that~$\LTimes_{1 \le i \le r} [D'_i,E'']^{\Ltimes m_i} = \sigma \cdot [D'',E'']$\;
\Return~$(M_{D''},M_{E''},\sigma \cdot \prod_{1 \le i \le r} (\nu_i \rho_i)^{m_i})$\;
\caption{Bilinear combination}
\label{alg:PBLC}
\end{algorithm}

\begin{rk}
The \textbf{for} loop in Algorithm~\ref{alg:PBLC} may be executed in parallel.
\end{rk}

\pagebreak

\section{Converting sums of points to Makdisi format}\label{sect:naive}

As we saw in Section~\ref{sect:MakRepPt}, in Makdisi's framework, a point on~$J$ is represented rather indirectly by a matrix encoding a subspace of the global section space~$V_2$ of~$\L^{\otimes 2}$. In contrast, we may well want to handle points of~$J$ provided to us as divisors on~$C$ of degree~$0$ and expressed as sums of points of~$C$. More generally, given two divisors~$D = \sum_i n_i P_i$,~$E = \sum_j m_j Q_j \in \Div^0(C)$ expressed as sums of points of~$C$, we may want to consider the point~$(D,E)$ of~$\Px$ (cf. notation~\eqref{eqn:(DE)}), or even the point
\begin{equation}
\Nm_D\left( \prod_j \big(t-t(Q_j)\big)^{-m_j} \right) \in \Px, \label{eqn:ptPnormt}
\end{equation}
for some rational function~$t$ on~$C$
, assuming of course that~$D$,~$E$, and~$t$ are such that these are well-defined elements of~$\Px$. For instance, this situation occurs in~\cite{QuadChab}. The purpose of this final section is to describe algorithms to convert such points of~$\Px$ into our format~\eqref{eqn:codepi}.

\subsection{Some extra assumptions: the vectors~$r_{n,P}$}

Since the points in the support of the divisors that we may face are arbitrary, we will assume from now on that for any point~$P \in C$, we are able to construct a local trivialisation~$\tau_P$ of~$\L$ at~$P$, and to compute a row vector~$r_{1,P} \in K^{\dim V_1}$ containing the values at~$P$ under~$\tau_P$ of the basis of the section space~$V_1$ of~$\L$ corresponding to the columns of the matrix~$M_{V_1}$. This means that we should retain some data describing this basis as we compute~$M_{V_1}$ for the first time during the initialisation of Makdisi's algorithms (cf. Section~\ref{sect:MakInit}).

\begin{ex}
In the situation described in Example~\ref{ex:OC(D)}, where~$C : f(x,y)=0$ is given to us by a (possibly singular) plane model and~$V_1$ is the Riemann-Roch space attached to some effective divisor~$D_0$, we should retain expressions in~$x,y$ for the basis of~$V_1$ used to calculate~$M_{V_1}$; for most~$P$, the row vector~$r_{1,P}$ can then be obtained simply by evaluating these expressions at~$P$ (some very manageable difficulties will occur if~$P$ corresponds to a singular point or to a point at infinity on the singular model~$f(x,y)=0$, and if~$P$ lies in the support of~$D_0$, then a trivialisation~$\tau_P$ must be constructed by multiplying by a suitable power of a local uniformiser at~$P$).
\end{ex}

\begin{ex}
In the situation described in Example~\ref{ex:Modular}, where~$C$ is a modular curve on which we represent points as elliptic curves equipped with the appropriate level structure, we should retain data describing the basis of~$V_1$ in terms of the modular forms~$\lambda_{v,w}$ mentioned in Example~\ref{ex:Modular}; for non-cuspidal~$P$, a trivialisation~$\tau_P$ can then be obtained by picking a Weierstrass model for the elliptic curve corresponding to~$P$, and when~$P$ is a cusp, a vector~$r_{1,P}$ can be obtained by computing the constant terms of~$q$-expansions at this cusp (cf.~\cite[Theorem 4.2.2]{MakMod} for explicit formulas).
\end{ex}

Furthermore, we will also assume that when we used the~$\DivAdd$ routine to compute the matrices~$M_{V_n}$ for~$n \le 5$, we retained the information required to express the elements of the corresponding basis of~$V_n$ as homogeneous polynomials of degree~$n$ in the elements of the basis of~$V_1$. It follows that we are actually able, given any point~$P \in C$, to compute for each~$n \le 5$ a row vector~$r_{n,P} \in K^{\dim V_n}$ containing the values under~$\tau_P$ of this basis of~$V_n$.

\begin{rk}
Since these vectors~$r_{n,P}$ depend on our choice of~$\tau_P$, they are actually only well-defined up to scaling. However, we will only use them for linear algebra calculations such as determining the hyperplane of~$V_n$ consisting of elements vanishing at~$P$, so this ambiguity will not matter. We also note that we will only require the~$r_{n,P}$ for~$n \le 4$, even though some calculations in Makdisi's framework take place in~$V_n$ for~$n$ as high as~$5$. 
\end{rk}

We will now explain how these row vectors~$r_{n,P}$ allow us to compute representations~\eqref{eqn:codepi} for points of~$\Px$ such as~\eqref{eqn:ptPnormt}. For the sake of simplicity, we limit ourselves to the case where the divisors~$D=P_1-P_2$ and~$E=Q_1-Q_2$ in~\eqref{eqn:ptPnormt} are differences of two points of~$C$; the general case can, of course, be handled thanks to the partial group laws algorithms described in Section~\ref{sect:Laws}.

We will also assume that the function~$t$ in~\eqref{eqn:ptPnormt} is a ratio~$t=v/w$ of sections~$v,w \in V_1$, so that the value~$t(P)$ of~$t$ at a point~$P \in C$ can be obtained from the vectors~$r_{1,P}$. Since~\eqref{eqn:bound_delta} ensures that~$\L$ is very ample, many ``interesting'' functions on~$C$ are of this form.

\subsection{Appetizer: Converting sums of points to~$J$}

It is instructive to begin with the simpler task of finding a Makdisi representation for the point~$[D] = [P_1-P_2]~$ of~$J$, given~$P_1, P_2 \in C$.

Recall from Section~\ref{sect:MakFramework} that~$\delta \defeq \deg \L$, and that we have fixed~$n_Z > 5 \delta$ points~$Z_1,\cdots,Z_{n_Z}$ on~$C$, and, given a matrix~$M$, denote by~$\Ker M$ a matrix whose columns form a basis of the right kernel of~$M$.

\begin{algorithm}[H]
\KwIn{Two points~$P_1,P_2$ of~$C$.}
\KwOut{A Makdisi representation of the point~$[P_1-P_2] \in J$.}
$I_\Delta \la$ a random subset of~$\{1,\cdots,n_Z\}$ of size~$\delta-1$\;\label{alg:convert_Mak:Delta}
\For{$j=1$ \KwTo~$2$}
{
$L_j \la$ the vertical concatenation of~$r_{2,P_j}$ and of the rows of~$M_{V_2}$ indexed by the~$i \in I_\Delta$\;
\If{$\rank L_j < \delta$}{Go back to line~\ref{alg:convert_Mak:Delta}\;\label{alg:convert_Mak:rank}}
$M_j \la M_{V_2} \times \Ker L_j$\;\label{alg:convert_Mak:MakSub}
}
Apply Makdisi's algorithm for subtraction in~$J$ to~$M_1$ and~$M_2$, and return the result\;
\caption{Putting points of~$J$ into Makdisi format}
\label{alg:convert_Mak}
\end{algorithm}

\begin{proof}
The subset~$I_\Delta$ corresponds to a divisor~$\Delta = \sum_{i \in I_\Delta} Z_i \in \Eff^{\delta-1}(C)$, which we ensure at line~\ref{alg:convert_Mak:rank} does not contain~$P_1$ nor~$P_2$. Therefore the matrices~$M_j$ represent the subspaces~$V_2(-P_j-\Delta)$ of~$V_2$, and thus encode valid points~$[\L^\dual(P_j+\Delta)]$ of~$J$. Subtracting them at line~\ref{alg:convert_Mak:MakSub} discards~$\L^\dual(\Delta)$ and retains~$[P_1-P_2]$. 
\end{proof}

\subsection{Converting sums of points to~$\Px$}

The final step of Algorithm~\ref{alg:convert_Mak} involves group operations in~$J$, and therefore some linear equivalences. In order to extend it to~$\Px$, it is thus necessary to retain information about these linear equivalences. We therefore begin by presenting a variant of Algorithm~\ref{alg:convert_Mak} where this group operation is dissected. As Algorithm~\ref{alg:convert_Mak}, it relies on a divisor~$\Delta \in \Eff^{\delta-1}(C)$ supported by some of the points~$Z_i$.

\begin{algorithm}[H]
\KwIn{Two points~$P_1,P_2$ of~$C$, and a matrix~$S_\Delta$ such that~$M_{V_2(-\Delta)} = M_{V_2} \times S_\Delta$ for some~$\Delta \in \Eff^{\delta-1}(C)$.}
\KwOut{Columns~$c_{s_2}, c_{s_3} \in K^{n_Z}$ representing sections~$0 \neq s_2 \in V_2$ and~$0 \neq s_3 \in V_3$, and a matrix~$M_D$ encoding~$D \in \EdC$ such that~${P_1}-{P_2} = D+(s_2)_2 - (s_3)_3 \in \Div(C)$.}
$M_{V_1(-{P_1})} \la M_{V_1} \times \Ker r_{1,{P_1}}$\tcp*{Represents~$V_1(-{P_1})$}
$M_{V_1(-{P_2})} \la M_{V_1} \times \Ker r_{1,{P_2}}$\tcp*{Represents~$V_1(-{P_2})$}
$r_{V_2(-\Delta) @ {P_1}} \la r_{2,{P_1}} \times S_\Delta$\tcp*{Values of basis of~$V_2(-\Delta)$ at~${P_1}$}
\If{$r_{V_2(-\Delta) @ {P_1}} = 0$}{\FAIL\tcp*{${P_1}$ meets~$\Delta$}}
$z \la$ random nonzero column vector in~$\Ker r_{V_2(-\Delta) @ {P_1}}$\;
$c_{s_2} \la M_{V_2} \times S_\Delta \times z$\tcp*{$(s_2)_2={P_1}+\Delta+E$,~$E \in \EdC$}
$M_{s_2 V_2} \la \FnMul(c_{s_2},M_{V_2})$\tcp*{$s_2 \cdot V_2 = V_4(-{P_1}-\Delta-E)$}
$M_{V_3(-\Delta-E)} \la \DivSub( M_{s_2 V_2}, M_{V_1(-{P_1})})$\tcp*{Represents~$V_3(-\Delta-E)$}
$M \la$ matrix such that~$M_{V_3(-\Delta-E)} = M_{V_3} \times M$\tcp*{Express subspace~$V_3(-\Delta-E)$ on basis of~$V_3$}
$r_{V_3(-\Delta-E) @ {P_2}} \la r_{3,{P_2}} \times M$\tcp*{Values of~$V_3(-\Delta-E)$ at~${P_2}$}
\If{$r_{V_3(-\Delta-E) @ {P_2}} = 0$}{\FAIL\tcp*{${P_2}$ meets~$\Delta+E$}}
$z \la$ random nonzero column vector in~$\Ker r_{V_3(-\Delta-E) @ {P_2}}$\;
$c_{s_3} \la M_{V_3(-\Delta-E)} \times z$\tcp*{$(s_3)_3={P_2}+\Delta+E+D$,~$D \in \EdC$}
$M \la \FnMul(c_{s_3},M_{V_2})$\tcp*{$s_3 \cdot V_2 = V_5(-{P_2}-\Delta-E-D)$}
$M \la \DivSub(M,M_{V_3(-\Delta-E)})$\tcp*{Represents~$V_2(-{P_2}-D)$}
$M \la \DivAdd(M,M_{V_1})$\tcp*{Represents~$V_3(-{P_2}-D)$}
$M_D \la \DivSub(M,M_{V_1(-{P_2})})$\tcp*{Represents~$V_2(-D)$}
\Return~$(c_{s_2},c_{s_3},M_D)$\;
\caption{Variant of Algorithm~\ref{alg:convert_Mak}}
\label{alg:Push1}
\end{algorithm}

\begin{proof}
Indeed,~$D = (s_3)_3 - (s_2)_2 +P_1-{P_2}$.
\end{proof}

%
%

Another difficulty is that in the setting described in Remark~\ref{rk:kpadic} where~$K$ is a~$p$-adic ring, we may face the situation where some of the points~$P_1,P_2,Q_1,Q_2$ are distinct but congruent mod~$p$. A naive evaluation of the function~$t$ in these circumstances would lead to performing~$p$-adic divisions that reduce to~$0/0 \bmod p$, which would thus incur loss of~$p$-adic accuracy. The following algorithm avoids this problem, by reexpressing functions on~$C$ as another ratio of sections of powers of~$\L$ and thereby solving such ``indeterminate forms''. The author thanks Davide Lombardo and Guido Lido for their help with inventing this trick.

\begin{algorithm}[H]
\KwIn{Integers~$l,r \le 4$ such that~$l +r > 5$, column vectors~$c_a, c_b$ representing sections~$0 \neq a \in V_{5-l}$,~$0 \neq b \in V_{5-r}$, and row vectors~$r_{n,P}$ and~$r_{n,Q}$ for~$n \le 4$ encoding two points~$P,Q \in C(K)$ such that~$\left(\frac{a}b \right)(P) / \left(\frac{a}b \right)(Q) \in \Kx$.}
\KwOut{The value~$\left(\frac{a}b \right)(P) / \left(\frac{a}b \right)(Q) \in \Kx$.}
$d_l \la l \delta + 1 - g$\tcp*{$\dim V_l$}
$d_r \la r \delta + 1 - g$\tcp*{$\dim V_r$}
$M_l \la$ matrix whose columns are~$\Set{c_a \odot c}{c \text{ column of } M_{V_l}}$\; 
$M_r \la$ matrix whose columns are~$\Set{c_b \odot c}{c \text{ column of } M_{V_r}}$\;
$M \la$ horizontal concatenation~$(M_l \vert M_r)$\;
$z \la$ a random nonzero column vector in~$\Ker M$\tcp*{$z \in K^{d_l + d_r}$}
$c_x \la$ vector formed of the first~$d_l$ coordinates of~$z$\;
$c_y \la$ vector formed of the last~$d_r$ coordinates of~$z$\;
$x_P \la r_{l,P} \times c_x$\;
$x_Q \la r_{l,Q} \times c_x$\;
$y_P \la r_{r,P} \times c_y$\;
$y_Q \la r_{r,Q} \times c_y$\;
\If{any of~$x_P,x_Q,y_P,y_Q$ vanishes}{\FAIL}
\Return~$\displaystyle \frac{y_P x_Q}{x_P y_Q}$\;
\caption{Solve indeterminate forms}
\label{alg:LinAlg}
\end{algorithm}

\begin{proof}
The matrices~$M_l$ and~$M_r$ represent respectively the linear maps
\[ \setmap{V_l}{V_5}{x}{a \cdot x} \quad \text{ and } \quad \setmap{V_r}{V_5}{y}{b \cdot y}, \]
so the matrix~$M$ represents the linear map
\[ \setmap{V_l \times V_r}{V_5}{(x,y)}{a \cdot x+b \cdot y}, \]
which cannot be injective since we have assumed that~$g \ge 1$ and~$l+r>5$. Therefore the column vectors~$c_x$ and~$c_y$ exist, and represent nonzero sections~$x \in V_l$ and~$y \in V_r$ such that~$a \cdot x+b \cdot y=0$, whence~$\left(\frac{a}b \right)(P) / \left(\frac{a}b \right)(Q) = \left(\frac{y}x \right)(P) / \left(\frac{y}x \right)(Q)$.
\end{proof}

%

By combining Algorithms~\ref{alg:Push1} and~\ref{alg:LinAlg}, we finally achieve an algorithm to convert points such as~\eqref{eqn:ptPnormt} into our representation~\eqref{eqn:codepi}, and which, when~$K$ is a~$p$-adic ring, manages not to lose~$p$-adic accuracy even when some of the points~$P_1,P_2,Q_1,Q_2$ are congruent mod~$p$ (assuming~\eqref{eqn:ptPnormt} does represent a well-defined point of~$\Px$):

\newpage

\begin{algorithm}[H]
\KwIn{Row vectors~$(r_{n,P_1})$,~$(r_{n,P_2})$,~$(r_{n,Q_1})$, and~$(r_{n,Q_2})$ for~$n \le 4$ representing four points~$P_1,P_2,Q_1,Q_2 \in C(K)$, and two vectors~$c_v,c_w \in V_1$ encoding the function~$t=v/w \in K(C)$.}
\KwOut{$\Nm_{P_1-P_2}\left( \frac{t-t(Q_2)}{t-t(Q_1)} \in \O_C(Q_1-Q_2) \right) \in \Px_{P_1-P_2,Q_1-Q_2}$.}
Use~$r_{1,Q_1}$ and~$r_{1,Q_2}$ to evaluate~$t(Q_1) \la \frac{v(Q_1)}{w(Q_1)}$ and~$t(Q_2) \la \frac{v(Q_2)}{w(Q_2)}$\;

\Repeat{lines~\ref{alg:Push:1},~\ref{alg:Push:mu}, and~\ref{alg:Push:norms} all succeed}{
\tcp{Pick random~$\Delta = \sum_{i \in I_\Delta} Z_i \in \Eff^{\delta-1}(C):$}
$I_\Delta \la$ random subset  of size~$\delta-1$ of~$\{ 1, \cdots, n_Z \}$\;
$M_{V_{2,\Delta}} \la$ matrix formed of the rows of~$M_{V_2}$ indexed by~$\Delta$\;
$S_\Delta \la \Ker M_{V_{2,\Delta}}$\tcp*{$M_{V_2(-\Delta)} = M_{V_2} \times S_\Delta$}


\tcp{Using Algorithm~\ref{alg:Push1} twice:}
Find~$c_{d_2}, c_{d_3}, M_{D}$ and~$c_{e_2}, c_{e_3}, M_{E}$, where~$d_2, e_2 \in V_2$,~$d_3, e_3 \in V_3$, and~$D,E \in \EdC$ are such that~$P_1-P_2 = D+(d_2)_2 - (d_3)_3$ and~$Q_1-Q_2 = E+(e_2)_2 - (e_3)_3$\;\label{alg:Push:1}


\tcp{Using Algorithm~\ref{alg:LinAlg} twice,}
\tcp{with~$l=3$,~$r=4$,~$c_a=c_{e_2}$, and~$c_b=c_{v} - t(Q_1) c_{w}$,}
\tcp{and then~$l=4$,~$r=2$,~$c_a=c_{v} - t(Q_2) c_{w}$, and~$c_b=c_{e_3}$:}
$\mu \la \left( \frac{e_2}{v - t(Q_1) w} \right) (P_1-P_2) \cdot \left( \frac{v - t(Q_2) w}{e_3} \right) (P_1-P_2)$\;\label{alg:Push:mu}


\tcp{Pick random~$s \in V_1$ such that~$s(P_1),s(P_2) \in \Kx$:}
\Repeat{$s_{P_1} \in \Kx$ and~$s_{P_2} \in \Kx$}
{
$c \la$ random column vector in~$K^{\dim V_1}$\;
$c_s \la M_{V_1} \times c$\tcp*{Represents random~$s \in V_1$}\label{alg:Push:s}
$s_{P_1} \la r_{1,P_1}\times c$\tcp*{$s(P_1)$}
$s_{P_2} \la r_{1,P_2} \times c$\tcp*{$s(P_2)$}
}
$\lambda_s \la s_{P_1} / s_{P_2}$\tcp*{$s(P_1)/s(P_2)$} \label{alg:Push:endt}


$M_{(s)_1} \la \FnMul(c_s,M_{V_1})$\tcp*{Represents~$s \cdot V_1 = V_2(-(s)_1)$}
$c_{d_0 d_2} \la c_{d_0} \odot c_{d_2}$\;
\tcp{Using Algorithm~\ref{alg:norm} ($n=3$ twice, then~$n=1$ twice):}
$\displaystyle \nu \la \left( \frac{d_2 d_0}{d_3} \right)(E) \cdot \left( \frac{d_3}{d_2 d_0} \right)((s)_1) \cdot \left( \frac{s}{e_0} \right)(D_0) \cdot \left(\frac{e_0}{s} \right)(D)$\;\label{alg:Push:norms}
}
$\lambda \la \mu \cdot \lambda_s \cdot \nu$\; \label{alg:Push:lambda}
\Return~$(M_D,M_E,\lambda)$\;
\caption{Converting sums of points to~$\Px$}
\label{alg:Push}
\end{algorithm}

\begin{proof}
Suppose first that~$P_1,P_2,Q_1,Q_2$, and the random objects such as~$d_2,D,s,\cdots$ generated along the algorithm are generic enough that none of the calculations below involve terms of the form~$0/0$.

We have
\[ P_1-P_2 = D+(d_2)_2-(d_3)_3 = D-D_0+(d_2 d_0 / d_3) \]
and similarly
\[ Q_1-Q_2 = E-E_0+(e_2 e_0 / e_3) \]
where~$d_2 d_0 / d_3$ and~$e _2 e_0 / e_3$ are rational functions on~$C$ since~$d_0,e_0 \in V_1$. Therefore, Lemma~\ref{lem:Poincare_compare_0} informs us that
\[ (P_1-P_2,Q_1-Q_2) = \left( \frac{d_2 d_0}{d_3} \right)(E-E_0) \cdot \left( \frac{e_2 e_0}{e_3} \right)(P_1-P_2) \cdot [D,E] \]
in the notation of~\eqref{eqn:(DE)} and~\eqref{eqn:[DE]}.

As a consequence, the point that we are trying to compute is
\begin{align*}
& \Nm_{P_1-P_2}\left(\frac{t-t(Q_2)}{t-t(Q_1)} \in \O_C(Q_1-Q_2) \right) \\
=& \left(\frac{t-t(Q_2)}{t-t(Q_1)}\right)(P_1-P_2) \cdot (P_1-P_2,Q_1-Q_2) \\
=& \lambda \cdot [D,E],
\end{align*}
where, since~$t=v/w$,
\begin{align*}
\lambda &= \left(\frac{v-t(Q_2) w}{v-t(Q_1) w}\right)(P_1-P_2) \cdot \left( \frac{d_2 d_0}{d_3} \right)(E-E_0) \cdot \left( \frac{e_2 e_0}{e_3} \right)(P_1-P_2) \\
&= \left(\frac{v-t(Q_2) w}{e_3} \big/ \frac{e_2}{v-t(Q_1) w}\right)(P_1-P_2) \cdot \left( \frac{d_2 d_0}{d_3} \right)(E-E_0) \cdot e_0(P_1-P_2) \\
&= \mu \cdot \left( \frac{d_2 d_0}{d_3} \right)(E-E_0) \cdot e_0(P_1-P_2).
\end{align*}
Even under our genericity assumption, the term~$e_0(P_1-P_2)$ may not be defined, since~$e_0$ was fixed at the construction of~$\Px$ and because the points~$P_1$ and~$P_2$ are imposed to us by the user. This is why we do not stop here, but instead introduce an extra random section~$s \in V_1$ at line~\ref{alg:Push:s}. We then have
\begin{align*}
\lambda &= \mu \cdot \left( \frac{d_2 d_0}{d_3} \right)(E-(s)_1) \cdot \left( \frac{d_2 d_0}{d_3} \right)(\overbrace{(s)_1-E_0}^{(s/e_0)}) \cdot e_0(P_1-P_2) \\
&= \mu \cdot \left( \frac{d_2 d_0}{d_3} \right)(E-(s)_1) \cdot \left( \frac{s }{e_0} \right)(\underbrace{P_1-P_2+D_0-D}_{(d_2 d_0/d_3)}) \cdot e_0(P_1-P_2)
\end{align*}
by Weil reciprocity, whence finally
\[ \lambda = \mu \cdot \underbrace{\left( \frac{d_2 d_0}{d_3} \right)(E-(s)_1) \cdot \left( \frac{s}{e_0} \right)(D_0-D)}_{\nu} \cdot \underbrace{s(P_1-P_2)}_{\lambda_s}, \]
which shows that Algorithm~\ref{alg:Push} is correct under the genericity assumption. By Zariski continuity, it is therefore correct whenever it succeeds, whence the outer \textbf{repeat~$\cdots$ until} loop.
\end{proof}

\begin{rk}
A similar (but slightly simpler) calculation shows that it is easy to alter Algorithm~\ref{alg:Push} so as to compute the point
\[ (P_1-P_2,Q_1-Q_2) \in \Px \]
(without reference to any function~$t$), by replacing lines~\ref{alg:Push:mu}--\ref{alg:Push:endt} with the following two lines:~$c_s \la$ random vector in the column span of~$M_{V_1}$, and then use Algorithm~\ref{alg:LinAlg} with~$l=r=2$,~$c_a = c_s \odot c_{e_2}$, and~$c_b = c_{e_3}$ to attempt to evaluate~$\displaystyle \mu \la \left( s e_2 / e_3 \right) (P_1-P_2)$; and finally, at line~\ref{alg:Push:lambda}, redefine~$\lambda \la \mu \cdot \nu$ (as~$\lambda_s$ is no longer defined).
\end{rk}

\end{document}